\documentclass{amsart}

\usepackage{graphicx}
\usepackage{amsfonts}
\usepackage{amsthm}
\usepackage{amsmath}
\usepackage{amssymb}
\usepackage{amsmath}
\usepackage{url}

\usepackage{natbib}
 \bibpunct[, ]{(}{)}{,}{a}{}{,}%
\usepackage{algorithm}
\usepackage{color}
\usepackage{algorithmicx}
\usepackage[noend]{algpseudocode}
\algnewcommand{\IfThenElse}[3]{% \IfThenElse{<if>}{<then>}{<else>}
  \State \algorithmicif\ #1\ \algorithmicthen\ #2\ \algorithmicelse\ #3}
\algnewcommand{\IfThen}[2]{% \IfThenElse{<if>}{<then>}
  \State \algorithmicif\ #1\ \algorithmicthen\ #2}
\algnewcommand{\algorithmicgoto}{\textbf{go to}}%
\algnewcommand{\Goto}[1]{\algorithmicgoto~#1}%

\usepackage[caption=false]{subfig}
\usepackage{multirow}
\usepackage{adjustbox}

\newtheorem{prop}{Proposition}
\newtheorem{corollary}{Corollary}
\theoremstyle{definition}

\theoremstyle{remark}

\newcommand{\LB}[0]{\mathrm{LB}}
\newcommand{\UB}[0]{\mathrm{UB}}

\begin{document}
%%%%%%%%%%%%%%%%

\title{Benders Adaptive-Cuts Method for Two-Stage Stochastic Programs}
\author{Cristian Ramírez-Pico}
\address{Faculty of Engineering and Sciences, Universidad Adolfo Ibáñez, Santiago, Chile} \email{cristian.ramirez@edu.uai.cl}
\author{Ivana Ljubi\'c}
\address{ESSEC Business School of Paris, 95021  Cergy-Pontoise, France} \email{ivana.ljubic@essec.edu}
\author{Eduardo Moreno}
\address{Faculty of Engineering and Sciences, Universidad Adolfo Ibáñez, Santiago, Chile} \thanks{Supported by ANID through grants Fondecyt 1200809 and STIC-AmSud STIC19007.}\email{eduardo.moreno@uai.cl}

\date{November 22, 2022}

\begin{abstract}
Benders decomposition is one of the most applied methods to solve two-stage stochastic problems (TSSP) with a large number of scenarios. The main idea behind the Benders decomposition is to solve a large problem by replacing the values of the second-stage subproblems with individual variables, and progressively forcing those variables to reach the optimal value of the subproblems, dynamically inserting additional valid constraints, known as Benders cuts. Most traditional implementations add a cut for each scenario (multi-cut) or a single-cut that includes all scenarios. In this paper we present a novel \emph{Benders adaptive-cuts method}, where the Benders cuts are aggregated according to a partition of the scenarios, which is dynamically refined using the LP-dual information of the subproblems. This scenario aggregation/disaggregation is based on the Generalized Adaptive Partitioning Method (GAPM), which has been successfully applied to TSSPs. We formalize this hybridization of Benders decomposition and the GAPM, by providing sufficient conditions under which an optimal solution of the deterministic equivalent can be obtained in a finite number of iterations. 
Our new method can be interpreted as a compromise between the Benders single-cuts and multi-cuts methods, drawing on the advantages of both sides, by rendering the initial iterations faster (as for the single-cuts Benders) and ensuring the overall faster convergence (as for the multi-cuts Benders).
Computational experiments on three TSSPs (the Stochastic Electricity Planning, Stochastic Multi-Commodity Flow and CVaR Facility Location) validate these statements, showing that the new method outperforms the other implementations of Benders method, as well as other standard methods for solving TSSPs, in particular when the number of scenarios is very large. 
Moreover, our study demonstrates that the method is not only effective for the risk-neutral decision makers, but also that it can be used in combination with the risk-averse CVaR objective.  
\end{abstract}

\keywords{Two-stage Stochastic Programming,  Benders Decomposition, Adaptive-Partition Method, Electricity Planning, Stochastic Multi-Commodity Flow, Conditional Value-at-Risk, Facility Location}
\maketitle

\section{Introduction}\label{sec:introducion}

In %Traditional 
two-stage linear stochastic programming (TSSP), a class of problems studied in this article,
%is one of the main classes of problems within stochastic optimization. In this type of problems %is structured to 
decisions are split between those made before and after the uncertainty of some modeling parameters is revealed \citep{birge2011introduction}. 
TSSPs are used for modeling many optimization problems that naturally appear in transportation, logistics, or telecommunications. Typical examples include stochastic multi-commodity flow problems, where 
   a set of given commodities has to be routed between different pairs of origin and destination nodes at minimum cost~\citep{gendron1999multicommodity,Barnhart2001}. In a more realistic setting, the demand of these commodities is uncertain, however decisions concerning the network structure including the capacity of its nodes/arcs have to be made \emph{here-and-now} (i.e., in the first-stage, before the actual demand is realized), whereas in the second-stage, once the demand is revealed, the routing of each commodity can be calculated in a \emph{wait-and-see} fashion~\citep{Sarayloo-et-al:2021,rahmaniani2018accelerating}.

TSSPs  are %type of problem is 
particularly challenging %to solve 
when the support of the uncertainty space is continuous and high-dimensional since there is usually no closed formula to express %formulation to represent 
 the second-stage optimization model. As an alternative, the sample average approximation (SAA) method~\citep{MAK199947, kleywegt2002sample} has been widely used to generate a deterministic equivalent formulation by sampling scenarios to obtain an approximately optimal solution using a (potentially large) number of discrete scenarios. 
  This large number of scenarios, which is needed for obtaining a reliable representation of data uncertainty, is a major challenge for solving practical applications of TSSPs using the SAA technique.
  
 Decomposition methods have shown strong capabilities in the design and implementation of efficient solution algorithms for TSSPs with a large number of scenarios. A large portion of work on these decomposition methods relies on Benders decomposition~\citep{benders1962partitioning}, which was further extended to stochastic programming by \citet{van1969shaped} and \citet{birge1988multicut} under the name of the L-Shaped method. This strategy  consists of executing three steps, namely, projecting, relaxing and linearizing some components of the TSSP, so that the original problem can be solved via an iterative procedure that utilizes the dualization of the projected terms. Dual extreme rays and extreme points of the second-stage subproblems are then used to generate valid inequalities, commonly known as %what are denominated 
feasibility and optimality cuts, respectively. %Both yield a set of constraints that, 
When included into the so-called master problem, these constraints result in an alternative valid formulation of the original optimization problem.
%However, 
We highlight the importance of an iterative method to identify only some of the extreme rays and extreme points since the enumeration of all of them is not practical in computational terms, and most of these extreme points and rays of the dual polyhedron are not active in an optimal solution.

A vanilla implementation of Benders decomposition %itself 
is usually not sufficient to solve  %the full problem 
large TSSPs directly. However, the idea of having a master problem and a set of subproblems solved separately, plus a set of strategies to improve and stabilize the solution obtained throughout the iterations of the algorithm, have made Benders decomposition method (or Benders method, for short) a successful technique. A detailed review on enhancing techniques for Benders methods is presented in~\citet{RAHMANIANI2017801}, from which we highlight   \emph{stabilization}~\citep{Rubiales2013,zaourar2014quadratic,fischetti16,fischetti17}, \emph{valid inequalities}~\citep{SAHARIDIS20116627}, \emph{cut removal} and \emph{cut selection}~\citep{pacqueau2012fast,yang2012tighter}, \emph{parallelization}~\citep{rahmaniani2019asynchronous}, and \emph{normalization}~\citep{fischetti2010}. 
%,jfcordeau2019,alfandari2021}.

In recent years, in the context of TSSPs, improvements to Benders  methods have focused on aggregating scenarios %, namely,
using different distance measures between scenarios. For example, \cite{vandenbussche2019data} use clustering techniques to aggregate scenarios, \cite{trukhanov2010adaptive} take those cuts that marginally contribute at a given iteration and group scenarios associated with them, and \cite{beltran2019fast} proposes aggregation based on conditional expectation over subsets of scenarios. Furthermore, \citet{crainic2021partial} propose to retain some important scenarios in the master problem and to  generate some artificial ones in order to derive strengthening valid inequalities. 
%that are being considered not too relevant based on a given criterion. 
Lastly, \cite{biel2019dynamic} generalize many of previous ideas and deliver a general framework to study different %Benders
scenario aggregations for Benders methods based on these %types of
distance measures. 

The idea of aggregating scenarios has been  explored not only in the context of Benders decomposition, but also
%to address 
for solving the deterministic equivalent formulation of two-stage stochastic programs.
In this context, dimension of the deterministic equivalent can be reduced by 
applying scenario 
aggregation using partitions of the scenario set. Such partitions are then iteratively refined to improve the obtained lower and upper bounds.
One of the first methods based on this idea is the \emph{sequential approximation method}~\citep{Birge1986, Frauendorfer1992}. In this method a partition of the scenarios is proposed using rectangular regions of random space. Then, a lower bound of the expected value of the second-stage subproblem is obtained using Jensen's approximation over the conditional expectation with respect to a partition of the scenarios, and an upper bound is obtained using generalizations of the {Edmundson-Mandasky inequality}~\citep{huang1977bounds}. An iterative refinement of the partition is proposed on the cell with the largest difference between these bounds, converging to the optimal solution. A more computationally efficient method is proposed in~\citet{pierre2011combined} replacing the computation of the bound with a variance-reducing  Monte Carlo estimator.
%directly.

Several recent approaches for aggregating scenarios in order to produce bounds for TSSPs are based on  measuring \emph{similarity} between scenarios. \citet{hewitt2021decision}
introduce the notion of \emph{opportunity costs} of predicting the wrong scenario, and use it to measure the similarity/distance between pairs of scenarios. 
%distance between two scenarios as the 
%difference in the objective value of the second-stage problem of the optimal solution found considering only one of the scenarios and the objective value of the second-stage problem considering only the second scenario (but considering the previous first-stage solution). 
The obtained distance measure is then embedded in a graph structure and graph clustering techniques are applied to partition the scenarios into smaller groups and to derive different upper and lower bounds for the TSSPs. 
%This cost-based similarity measure outperforms other metrics based only on the random outcomes, and other classic clustering techniques~\citep{keutchayan2021problem}.
\citet{keutchayan2021problem} seek to find a clustering of the scenario set and a representative of each cluster, so that these representatives can be used to solve a smaller approximative TSSP. They  introduce a \emph{discrepancy
   measure} which assesses how well a representative scenario within the cluster matches the average cost.
%   {[}2{]} to directly tackle the scenario reduction problem when solving
%   two-stage stochastic programs. The authors 
   The goal is to find a partition of the set of scenarios into $K$ clusters (and a representative of each cluster) so that the discrepancy is minimized. 
   %latter being done as a means to also minimize the overall    implementation error caused by the reduction that is then applied    (i.e., choosing to replace each cluster by its representative).
%   Another recent example of how partitions can be used to reduce the set
%   of scenarios can be found in {[}3{]}. 
    Finally,  \citet{Bertsimas-Mundru:2022} introduce the  concept of problem-dependant divergence as a means to evaluate the difference between scenarios and propose an algorithm to partition the
   scenario set and solve the scenario reduction problem. They also provide conditions under which  their approach for scenario reduction reproduces the SAA results.

Another aggregation method similar to the sequential approximation method but with a different refinement of the partition of scenarios is proposed by \citet{espinoza2014primal} for a risk-minimization portfolio problem. The authors  %propose to separate
group scenarios and disaggregate them by considering only those that are critical %scenarios 
for the risk measure, i.e., those that have different dual values for a given portfolio solution. Later, \citet{song2015adaptive} generalize this concept in the so-called   \emph{adaptive partition method} (APM), ensuring the existence of a sufficient partition for any two-stage stochastic program with fixed recourse and discrete uncertainty space. % with an algorithm that converges in finite steps to the optimal solution of the original problem. 
More recently, \citet{ramirez2021generalized} develop  a generalization of the previous algorithm called the \emph{generalized adaptive partition method} (GAPM), proving the existence of a \emph{finite and sufficient} partition and providing an implementation of the aggregation technique in a more general setting with a continuous uncertainty set. In the simplest case, the idea behind GAPM is to aggregate all scenarios to obtain a smaller and easier master problem, and then iteratively solve the subproblems to disaggregate the scenarios into subsets sharing, e.g., the same dual optimal solutions. Contrary to the similarity-based methods mentioned above, where heuristic partitions and valid bounds are provided, the GAPM guarantees to find an optimal solution of the TSSP using a potentially smaller number of aggregated scenarios.

\paragraph{Our Contribution.}
%In this paper, 
Following the theory of the GAPM, we formalize the idea of aggregating scenarios  in the Benders decomposition context.  
%obtained by aggregating the original ones. 
%that is, 
%aggregating all scenarios to obtain aggregated Benders cuts and then iteratively disaggregating these cuts based on the dual optimal solutions of the subproblems. 
%
%As mentioned earlier, \citet{pay2017partition} already explored this approach; yet the convergence, formal proof of valid cuts and extended numerical experiment are not available.
%In this context, 
The purpose of this paper is twofold: (1)~to combine two methodologies, i.e., Benders decomposition and the GAPM, that have proven to be successful for solving two-stage stochastic problems and (2) to provide sufficient conditions under which a given set of scenarios can be aggregated without sacrificing the optimality. % solution. 
We develop the underlying theory that enables to apply the Benders method to a smaller set of aggregated  scenarios, based on the dual optimal solutions of the subproblems. We name this new approach the \emph{Benders adaptive-cuts method}. 
Due to aggregation, smaller TSSPs are obtained, thus, 
%to obtain an optimal solution of the problem in the context of Benders decomposition, 
impacting the number of cuts that are needed to solve the problem. 
The aggregated scenario cuts resemble the single-cuts Benders method, the disaggregated ones resemble the multi-cuts Benders method, and using a guided scenario refinement procedure (stemming from the GAPM), we manage to draw on the advantage of both worlds. 
%In the existing literature on the GAPM, the TSSPs are required to have a relatively complete recourse. We extend this theory and
We also show how to deal with infeasible subproblems, thus requiring only a fixed recourse for the TSSPs. 

In our computational study, we focus on three practically relevant examples of stochastic network flow problems, for which Benders decomposition has been successfully implemented in the past.
These problems are: Stochastic Electricity Planning, 
Stochastic Multi-Commodity Flow, and Stochastic Facility Location with the CVaR objective. 
The obtained computational results support our theoretical findings, demonstrating that the Benders adaptive-cuts method outperforms the other standard implementations of Benders decomposition, as well as the GAPM or the deterministic equivalent formulation. The computational advantages of the new method are particularly pronounced for TSSPs with a large number of scenarios.

To the best of our knowledge, the only time the idea of using APM for Benders decomposition has been mentioned is in~\cite{pay2017partition}. In this paper, the authors propose to generate \emph{coarse} optimality cuts
to improve the computational performance of their Benders decomposition implementation. 
%generating optimality-cuts in the initial iterations.
These cuts are generated from an adaptive partition of the scenarios, but are not used as a stand-alone approach, possibly due to the lack of theoretical arguments for its convergence.    
%into the standard single-cut Benders method, as a way } 
Our paper closes this gap and provides formal theoretical arguments for hybridizing (G)APM with Benders decomposition, including a derivation of feasibility and optimality adaptive cuts %derived 
from a partition of the scenarios and a proof of sufficient conditions needed to obtain the optimal solution. Integration of the APM with other decomposition methods is proposed by \cite{van2017adaptive} for level decomposition, and \cite{siddig2019adaptive} for the SDDP algorithm for multistage stochastic problems.

The remainder of this paper is organized as follows. %A 
The problem statement is presented in Section~\ref{sec:statement} along with the overview of two standard implementations of the Benders method, based on single-cuts and multi-cuts, respectively. Our theoretical framework for integrating the GAPM into the Benders method and a generic implementation procedure are given in Section~\ref{sec:bendersMethod}. Section~\ref{sec:BendersStochNetworks} presents three classical stochastic network flow optimization problems for which we applied  the new methodology, and Section~\ref{sec:computational} shows the results of our computational experiments. Finally, we draw some final conclusions in Section~\ref{sec:conc}.

\section{Problem Statement and Mathematical Framework}\label{sec:statement}

In this section, we provide a formal %mathematically
definition of two-stage linear stochastic programs that are  subject of this study. We then provide an overview 
%detailed explanation 
of two classical implementations of the Benders method based on a separation of multi-cuts and single-cuts, respectively, before we present the theoretical framework of the novel Benders adaptive-cuts approach.% as well as a general algorithm that guarantees convergence in finite iterations for any of the three implementations described.

\subsection {%Definition and 
Problem Formulation} 
We study the following two-stage linear stochastic program   with fixed recourse:
\begin{equation}
    \min\limits_{x\in\mathcal{X}} c^\intercal x + \sum_{s\in \mathcal{S}} p^s Q(x,\xi^s)
    \label{eq:main}
\end{equation}
where $x$ is a first-stage variable to which we associate a cost vector $c \in \mathbb{R}^n$, and $\mathcal{X}\subseteq\mathbb{R}^n$ is a non-empty closed (polyhedral) set describing feasible  first-stage solutions. The uncertainty is modeled using a discrete sample space which is composed of a set of scenarios $\mathcal{S}$,  each with associated probability $p^s>0$ for $s\in \mathcal{S}$. In practice, the set of scenarios is composed by equally probable scenarios sampled from a more complex distribution, as  presented in the sample average approximation method \citep{MAK199947, kleywegt2002sample}. 
For a given realization $\xi^s:=(T^s,h^s)$ of the random variable $\xi$ from $\mathcal{S}$, $Q(x,\xi^s)$ is the associated second-stage subproblem, defined as the following linear program (LP) :
\begin{subequations}
    \begin{align}
Q(x,\xi^s) := \min_{y \ge 0}&\  q^\intercal y \\
	 W y &= h^s - T^{s} x
%	y &\geq 0
\end{align} 
\label{eq:subproblem}
\end{subequations}
In this LP, $y\in \mathbb{R}^{m}$ is a 
%the 
second-stage variable, $W\in \mathbb{R}^{p \times m}$ is a fixed recourse matrix, $q\in\mathbb{R}^m$ is a deterministic cost vector. A random technology matrix $T^s \in \mathbb{R}^{p\times n}$ and a random right-hand side vector $h^s\in \mathbb{R}^p$ come from the discrete sample space {$\mathcal{S}$.}
%composed of a set of scenarios $\mathcal{S}$,  each with probability $p^s$ for $s\in \mathcal{S}$. 
We assume that there exists some $\bar{x}$ such that $Q(\bar{x},\xi^s)$ is feasible and bounded for all $s\in \mathcal{S}$.
Most commonly, first-stage decision space corresponds to a polytope $\mathcal{X}=\{x\in\mathbb{R}^n : Ax=b\}$, however it can comprise more complex structures, including constraints that restrict the values of (some of) $x$ variables to integer numbers.
A \emph{deterministic equivalent} of the problem \eqref{eq:main} is obtained after introducing a copy of the second-stage variable $y^s \in \mathbb{R}^m$, for each scenario $s \in \mathcal{S}$:
    \begin{align}
\min_{x \in \mathcal{X}, y^s \geq 0}   \{ c^T x + \sum_{s \in \mathcal{S}} p^s q^\intercal y^s \; 	\mid \;  W y^s &= h^s - T^{s} x, s \in \mathcal{S} \}
\label{eq:DE}
\end{align}

\subsection{Classical Way(s) of Deriving Benders Cuts}\label{sec:bendersClassic}

In this section we briefly review the classical Benders decomposition approach created to address large TSSPs. This method is also known as the \textit{L-Shaped method}. Its main idea is to solve a large TSSP by replacing the values of second-stage subproblems with single variables. The algorithm progressively forces these variables to reach the optimal value of the subproblems, by inserting in a dynamic  fashion additional valid %artificial 
 constraints, known as \emph{Benders cuts.}

To present the method, we rewrite problem~\eqref{eq:main} using the value-function reformulation:

\begin{subequations} 
    \begin{align}
    \min\limits_{x\in\mathcal{X}, \theta \in \mathbb{R}^{|\mathcal{S}|}}~~ & c^\intercal x + \sum_{s\in \mathcal{S}} p^s \theta^s \\
    &Q(x,\xi^s) \leq \theta^s \qquad  %\forall~ 
    s\in \mathcal{S}.
    \end{align}
    \label{GenProblem}
\end{subequations}
The value of $Q(x,\xi^s)$ can be replaced by the optimal solution of its dual problem (assuming that the problem is well-defined and the optimal solution exists):

\begin{subequations}
% \begin{align}
% Q(x,\xi^s) = \max_{\lambda^s \in \mathbb{R}^p}\ &(h^s-T^s x)^\intercal \lambda^s  \\
% 	&W^\intercal \lambda^s \leq q   \label{dual_space}
% \end{align}     
\begin{align}
Q(x,\xi^s) = \max_{\lambda  \in \mathbb{R}^p}\ &(h^s-T^s x)^\intercal \lambda   \\
	&W^\intercal \lambda \leq q.   \label{dual_space}
\end{align}     
\label{Scen:Dual}
\end{subequations}
In general, let $\Lambda = \{\lambda\in \mathbb{R}^p: W^\intercal \lambda \leq q\}$ be the feasible region of the dual problem \eqref{Scen:Dual} (notice that $\Lambda$ does not depend on $x$). Assuming that the problem has a relatively complete recourse, $\Lambda$ is not empty, so let $\mathrm{XP}(\Lambda)$ and $\mathrm{XR}(\Lambda)$ be the sets of its extreme points and extreme rays, respectively. Hence, we can reformulate problem~\eqref{eq:main} as

\begin{subequations}
    \begin{alignat}{2}
 \min_{x\in\mathcal{X} ,\theta \in \mathbb{R}^{|\mathcal{S}|}}~& c^\intercal x + \sum_{s \in \mathcal{S}} p^s \theta^s\\
%  	(h^s-&T^s x)^\intercal \hat\lambda^s \leq \theta^s && \quad  \forall~ \hat{\lambda}^s \in \mathrm{XP}(\Lambda),~\forall ~ s\in \mathcal{S} 	\label{optcut_m}\\
%     (h^s-&T^s x)^\intercal \tilde\lambda^s \leq 0 && \quad \forall~ \tilde{\lambda}^s \in \mathrm{XR}(\Lambda),~\forall ~ s\in \mathcal{S}	\label{feascut_m} 
	(h^s-&T^s x)^\intercal \hat\lambda  \leq \theta^s && \quad    \hat{\lambda}  \in \mathrm{XP}(\Lambda),~s\in \mathcal{S} 	\label{optcut_m}\\
    (h^s-&T^s x)^\intercal \tilde\lambda  \leq 0 && \quad   \tilde{\lambda}  \in \mathrm{XR}(\Lambda),~s\in \mathcal{S}	\label{feascut_m} 
\end{alignat}
\label{Benders_multicut}
\end{subequations}
In this model, constraints~\eqref{optcut_m} are known as \emph{Benders optimality cuts} -- they ensure that $\theta^s$ is a lower bound for  $Q(x,\xi^s)$ for any extreme point of $\Lambda$, so, in particular, it is a bound for its maximum value. In the  case that $Q(\hat{x},\xi^s)$ is infeasible for a given $\hat{x}$, its dual is unbounded, so constraints~\eqref{feascut_m} (also known as \emph{Benders feasibility cuts}) ensure that $\hat{x}$ is no longer considered a feasible solution. This problem reformulation contains an exponential number of constraints, but it can be solved using a cutting plane approach by  iteratively selecting a candidate solution $\hat{x}$ of the \emph{restricted master problem} (which is the problem \eqref{Benders_multicut} with a subset of constraints \eqref{optcut_m}-\eqref{feascut_m}) and solving the subproblems $Q(\hat{x},\xi^s)$ to add possibly violated feasibility or optimality cuts on the fly. 
This method is known as the \emph{multi-cuts Benders method} (or multi-cuts L-shaped method) in stochastic programming, and it converges in a finite number of iterations~\citep{birge1988multicut}.

Alternatively, it is possible to represent the total second-stage cost using a single variable $\Theta$, by  aggregating  optimality cuts \eqref{optcut_m} into a single Benders optimality cut of type \eqref{optcut_s}.  That way, we obtain the following equivalent problem reformulation: 
\begin{subequations}
    \begin{alignat}{2}
 \min_{x\in \mathcal{X} ,\Theta \in \mathbb{R}}~ &c^\intercal x + \Theta\\
	\sum_{s\in S} p^s (h^s-&T^s x)^\intercal \hat{\lambda}^s \leq \Theta&& \quad   (\hat{\lambda}^1,\dots,\hat{\lambda}^{|\mathcal{S}|}) \in \mathrm{XP}(\Lambda)^{|\mathcal{S}|} 	\label{optcut_s}\\
	(h^s-&T^s x)^\intercal \tilde{\lambda}  \leq 0&& \quad    \tilde{\lambda}  \in \mathrm{XR}(\Lambda),  s\in \mathcal{S}	\label{feascut_s} 
\end{alignat}
\label{Benders_singlecut}
\end{subequations}
in which $(\hat{\lambda}^1,\dots,\hat{\lambda}^{|\mathcal{S}|})$ corresponds to a Cartesian product of $|\mathcal{S}|$ extreme points of $\mathrm{XP}(\Lambda)$.  
Problem~\eqref{Benders_singlecut} was considered in the original formulation of the L-shaped method~\citep{van1969shaped}, also known as the \textit{single-cuts Benders method}.
A notable difference between the single-cuts and the multi-cuts variant is in the separation of optimality cuts: while multi-cuts can be separated for each feasible scenario independently, to separate constraint \eqref{optcut_s}, the first-stage solution $x$ must yield feasible subproblems $Q(x,\xi^s)$ for all scenarios $s \in \mathcal{S}$.

The advantages of the single-cuts reformulation include a smaller number of variables and potentially a smaller number of cuts separated during the cutting plane algorithm. Thus, solving the master problem at each cutting plane iteration is faster than solving the problem via the multi-cuts implementation. An important disadvantage is that optimality cuts \eqref{optcut_s} can be added only when the whole set of subproblems is feasible.
Moreover, at any given iteration, the cuts obtained from this implementation are weaker than those generated via the multi-cuts implementation, which explains why  
%On the other hand, 
the multi-cuts implementation %has benefits such as convergence 
may converge in a fewer number of iterations. 
In addition, for the latter one, it is possible to add optimality cuts for only a subset of scenarios, saving time due to the smaller number of subproblems solved. In general, however, the multi-cuts approach entails additional cost of adding many more cuts to the master problem, thus making the master problem larger and slower to solve. 

\section{GAPM and Benders Adaptive-Cuts}
\label{sec:bendersMethod}
We start this section by briefly summarizing the major ideas behind the generalized adaptive partition method, from which we then derive Benders adaptive-cuts. 

\subsection{The Generalized Adaptive Partition Method}
In Section \ref{sec:introducion}, we briefly mentioned some of the techniques that aim to reduce the size of stochastic two-stage programs. %Most of the existing approaches use the notion of distance between scenarios to group similar scenarios~\citep[see][]{crainic2014scenario, biel2019dynamic}. Nonetheless, a different  framework 
The Generalized Adaptive Partition Method has been developed %in the GAPM
by \citep{espinoza2014primal, song2015adaptive, ramirez2021generalized}, where the authors provide formal conditions under which any TSSP with fixed recourse can be solved exactly using a smaller deterministic equivalent reformulation of problem \eqref{eq:main}. We briefly summarize the major ideas behind the GAPM and introduce the necessary notation.

% More precisely, let be a two-stage stochastic problem with fixed-recourse 
% \begin{equation}
%     \min\limits_{x\in\mathcal{X}} c^\intercal x + \mathbb{E}[Q(x,\xi)]
%     \label{eq:mainOrig}
% \end{equation}
% where $\xi$ is a random vector in the probability space $(\Omega, \mathcal{A}, \mathbb{P})$ containing the random elements $\{h^\xi, T^\xi$\} of the second-stage subproblem~\eqref{eq:subproblem}. 
Given a two-stage stochastic program with fixed 
recourse as defined in  problem \eqref{eq:main}, 
let us assume that the subproblems are feasible for any $x \in \mathcal{X}$ and any realization of uncertain parameters (i.e., we have a \emph{relatively complete recourse}). Let $\mathcal{P}$ be a partition of the uncertainty space $\Omega$. Note that $\Omega$ can be a discrete set (like $\mathcal{S}$ in our previous problem) or continuous. 
For each element $P \in \mathcal{P}$ of this partition,  let $h^P=\mathbb{E}[h^\xi|P]$ and $T^P=\mathbb{E}[T^\xi|P]$ be the conditional expectations, given $P$, of  the right-hand side vector and the technology matrix, respectively. Consequently, we can define the aggregated subproblem associated to element $P \in \mathcal{P}$ as follows:
\[ Q(x,\mathbb{E}[\xi|P]) = \min\left\{ q^\intercal y \; \mid \; W y = h^P - T^Px, y\geq 0\right\}. \]
\citet{ramirez2021generalized} provide conditions on the dual of this aggregated subproblem to ensure its equivalence with the conditional expectation of $Q(x,\xi)$ given~$P$, denoted as $\mathbb{E}\big[ Q({x},\xi) | P \big]$.
\begin{prop}[\citet{ramirez2021generalized}]
		Let $\bar{x}\in\mathcal{X}$ and $P\subseteq \Omega$ be such that $Q(\bar{x},\xi)$ is feasible for all $\xi\in P$, and let $\hat{\lambda}^\xi$ be optimal solutions of the LP-dual of the subproblem $Q(\bar{x},\xi)$. If $\hat{\lambda}^\xi$ for $\xi\in P$ satisfies
		\begin{subequations}
			\label{HypoLemmaGen}
			\begin{eqnarray}
				\label{HypoLemmaGen1}
				\Big(\mathbb{E}\big[ h^\xi | P \big]\Big) ^\top \Big( \mathbb{E}\big[ \hat{\lambda}^\xi | P \big]\Big) &=& \mathbb{E}\left[ \left.h^\xi\right.^\top \hat{\lambda}^\xi \Big|  P \right]\\
				\label{HypoLemmaGen2}
				\bar{x}^\top\Big(\mathbb{E}\big[ T^\xi | P \big]^\top \mathbb{E}\big[ \hat{\lambda}^\xi | P \big]\Big) &=& \bar{x}^\top \mathbb{E}\left[ \left.T^\xi\right.^\top \hat{\lambda}^\xi \Big| P \right]	
			\end{eqnarray}
		\end{subequations}
	then
		\[ Q\big(\bar{x},\mathbb{E}\left[\xi|P\right] \big) = \mathbb{E}\big[ Q(\bar{x},\xi) | P \big]. \]
		\label{GenLemma}
\end{prop}

From this result, given a partition $\mathcal{P}$ of $\Omega$ such that its elements satisfy the conditions~\eqref{HypoLemmaGen}, by the law of total expectation we can rewrite
\[ \mathbb{E}[Q({\bar x},\xi)] = \sum_{P\in\mathcal{P}} \mathbb{E}[Q({\bar x},\xi)|P]\cdot \mathbb{P}(P) =  \sum_{P\in\mathcal{P}} Q({\bar x},\mathbb{E}[\xi|P])\cdot \mathbb{P}(P).\]
%\sout{Therefore, an optimal solution of $\min_{x\in \mathcal{X}}\{c^\intercal x +\sum_{P\in\mathcal{P}} Q(x,\mathbb{E}[\xi|P])\cdot \mathbb{P}(P)\}$ is also optimal %for~\eqref{eq:mainOrig} for~\eqref{eq:main}.} 
Therefore, if $\bar x \in \arg \min_{x\in \mathcal{X}}\{c^\intercal x +\sum_{P\in\mathcal{P}} Q(x,\mathbb{E}[\xi|P])\cdot \mathbb{P}(P)\}$ for a partition $\mathcal{P}$ whose all elements $ P \in \mathcal{P}$ satisfy the condition \eqref{HypoLemmaGen}, then $\bar x$ is also optimal %for~\eqref{eq:mainOrig}.
for~\eqref{eq:main}. 
Moreover, there always exists a finite partition $\mathcal{P}$ satisfying the required conditions~\eqref{HypoLemmaGen}. Indeed, for a discrete uncertainty space, this is true at least for $\mathcal{P}= \mathcal{S}$. For continuous uncertainty space, see \cite{ramirez2021generalized} for more details. 
The  deterministic equivalent reformulation with respect to $\mathcal{P}$ is given as:
\begin{align}
\min_{x \in \mathcal{X}}   \{ c^T x + \sum_{P \in \mathcal{P}} p^P q^\intercal y^P \; 	\mid \;  W y^P &= h^P - T^{P} x, \; y^{P} \geq 0, \; P \in \mathcal{P} \}
\label{eq:DE_P}
\end{align}

In what follows, we propose a Benders adaptive-cuts method which aims to solve the problem \eqref{eq:DE_P} in a cutting plane fashion in which  the scenario disaggregation is guided according to the GAPM scheme. While Proposition \ref{GenLemma} assumes that for a given $\bar x$ all subproblems are feasible, %relatively complete recourse, 
we also show how to deal with a more general setting, in which certain realizations of $x$ may render the subproblems infeasible.
% i.e., we do not impose the relatively complete recourse condition for our method.

%The authors in \cite{ramirez2021generalized} present sufficient conditions on dual variables of subproblems~\eqref{Scen:Dual} to ensure that the cumulative costs of scenarios belonging to the same group (solved separately) are equal to the objective of one single aggregated subproblem that is parametrized based on the conditional expectation of random components.

\subsection{Benders Adaptive-Cuts}

The GAPM leads to the idea of reformulating the TSSPs using Benders decomposition with a potentially smaller set of scenarios. This results into what we call a \emph{Benders adaptive-cuts method}, in which we generate feasibility and optimality cuts at each iteration while avoiding an increase in the size of the master problem, as often occurs in the multi-cuts implementation. Our method aims to benefit from the advantages of both, the Benders multi-cuts and single-cuts methods. The reformulation of the original stochastic program \eqref{eq:main} is similar to \eqref{Benders_multicut}, but injecting fewer cuts than in the multi-cuts case. We add more optimality cuts compared to the single-cuts implementation,
with the advantage of addressing both optimality and feasibility  during the same iteration, which is not the case when we refer to the single-cuts implementation.

We first show how to aggregate a set of optimality cuts for a subset $P \subset \mathcal{S}$ of scenarios and demonstrate the validity of this new cut for reformulation \eqref{Benders_multicut}.

\begin{prop} \label{prop_adapt_valid_cut}
Let $P\subseteq \mathcal{S}$, and let  $p^P = \sum_{s\in P} p^s$, $h^P = \frac{\sum_{s\in P} p^s h^s}{\sum_{s\in P} p^s}$ and $T^P = \frac{\sum_{s\in P} p^s T^s}{\sum_{s\in P} p^s}$.  Then, the following inequalities 
%$x$ and let $\hat{\lambda}^s$  be the dual optimal solution of the subproblems $Q(x,\xi^s)$ for all $s\in S$. Let $\lambda^P = \frac{\sum_{s\in P} p^s {\lambda}^s}{\sum_{s\in P} p^s}$. Then, the inequality
\begin{equation}
%p^P \left[\sum_{k\in \mathbb{K}} \left({d_k^P}^\intercal \hat{\lambda}_k^P \right) + (h-Tx)^\intercal \hat{\mu}^P\right]  \leq \sum_{s\in P} p^s\theta^s
p^P \cdot (h^P-T^Px)^\intercal \hat{\lambda}^P  \leq \sum_{s\in P} p^s\theta^s  \qquad   \hat{\lambda}^P\in\mathrm{XP}(\Lambda)\label{eq:aggregatedBender}
\end{equation}
are valid for problem \eqref{Benders_multicut}.
\end{prop}

\begin{proof}
For a given $x \in \mathcal{X}$ and $s\in \mathcal{S}$, let $\hat{\lambda}^s$  be an optimal solution of the subproblem $Q(x,\xi^s)$ (assuming the associated subproblem is feasible). Then, $(\hat{\lambda}^s)_{s\in P}$ %$(\hat{\lambda}^s,\hat{\mu}^s)_{s\in P}$
is also an optimal solution for the problem
\begin{align*}
 %\red{\sout{\max}} 
 \sum_{s\in P} p^s Q(x,\xi^s)  = \max_{{\lambda^s \in \mathbb{R}_+^{|P|}}}& \sum_{s\in P} p^s \cdot (h^s-T^s x)^\intercal \lambda^s  \\  
  & W^\intercal \lambda^s \leq q \qquad  s\in P.
%  \max \sum_{s\in P} p^s Q^D(x,s)  = \max& \sum_{s\in P} p^s \left[\sum_{k\in \mathbb{K}} {d^s_k}^\intercal \lambda^s_k + (h-Tx)^\intercal \mu^s  \right]\\  
%  	& \mathcal{N}^\intercal \lambda_k^s + W^\intercal \mu^s \leq q_k \qquad  \forall ~k\in \mathbb{K}, \forall~ s\in P
\end{align*}
On the other hand, for any $\hat{\lambda}^P\in\mathrm{XP}(\Lambda)$, since $\hat{\lambda}^P$ is a feasible solution for the dual of $Q(x,\xi^s)$ for all $s\in \mathcal{S}$, by optimality of $(\hat{\lambda^s})_{s\in P}$, we have:
\begin{align*}
\sum_{s\in P} p^s \cdot (h^s-T^s x)^\intercal \hat{\lambda}^s  & \geq  \sum_{s\in P} p^s \cdot (h^s-T^s x)^\intercal \hat{\lambda}^P \\
& = \left(\left(\sum_{s\in P} p^s{h^s}\right)- \left(\sum_{s\in P} {p^s} T^s\right) x\right)^\intercal \hat{\lambda}^P \\
& = p^P \cdot (h^P-T^P x)^\intercal \hat{\lambda}^P
\end{align*}
Since the optimality cuts for the multi-cuts case enforce that 
\[ (h^s-T^s x)^\intercal \hat{\lambda}^s \leq \theta^s  \qquad  s\in \mathcal{S}, \]
then by aggregating these cuts for all $s\in P$, we obtain 
\begin{align*}
     p^P \cdot (h^P-T^P x)^\intercal \hat{\lambda}^P
\leq \sum_{s\in P} p^s (h^s-T^s x)^\intercal \hat{\lambda}^s  
\leq \sum_{s\in P} p^s \theta^s
\end{align*}
so~\eqref{eq:aggregatedBender} is a valid cut for problem \eqref{Benders_multicut}.
\end{proof}

Note that $h^P$ and $T^P$ can be interpreted as weighted averages of the random components $h^s$ and $T^s$ among all scenarios $s\in P$, and $p^P$ is the aggregated probability of these scenarios. Hence, \eqref{eq:aggregatedBender} can be seen as an aggregation of the optimality cuts~\eqref{optcut_m} from the original multi-cuts Benders method. 
This type of aggregated Benders cuts has been also studied as a way to \emph{a priori} create artificial scenarios to tighten the master problem~\citep{crainic2021partial}, or to generate Benders cuts~\citep{rahmaniani2021asynchronous}, for a predefined clustering of scenarios.

% Next, we present conditions that must satisfy a partition $\mathcal{P}$ of the scenarios to obtain the optimal solution of the problem by adding only these aggregated cuts. These conditions are equal to the conditions required for GAPM presented in Proposition~\ref{GenLemma}.

{Next, we present sufficient conditions under which   Benders optimality cuts imposed with respect to a partition $\mathcal{P}$ of $\mathcal{S}$ provide a proper problem reformulation, and thus, an optimal solution is guaranteed to be found.}
%by solving a smaller adding only Benders aggregated cuts with respect to $\mathcal{P}$. 
These conditions %are equal 
{correspond} to the conditions required for the GAPM (cf.\  Proposition~\ref{GenLemma}).

\begin{prop} \label{prop_optimal}
Let $x^*$ be an optimal solution of the problem
\begin{subequations}
    \begin{alignat}{2}
 \min_{x\in \mathcal{X} ,\theta \in \mathbb{R}^{|\mathcal{S}|}} ~&c^\intercal x + \sum_{s \in \mathcal{S}} p^s \theta^s\\
% p^P (h^P-&T^P x)^\intercal \hat{\lambda}^P \leq \sum_{s\in P} p^s\theta^s  &&\quad 
% %\forall ~
% \hat{\lambda}^P\in \mathrm{XP}(\Lambda),  P\in \mathcal{P} \label{eq:aggregatedBender_duplicated}\\
p^P (h^P-&T^P x)^\intercal \hat{\lambda} \leq \sum_{s\in P} p^s\theta^s  &&\quad 
%\forall ~
\hat{\lambda} \in \mathrm{XP}(\Lambda),  P\in \mathcal{P} \label{eq:aggregatedBender_duplicated}\\
    (h^s-&T^s x)^\intercal \tilde\lambda  \leq 0 && \quad 
    %\forall~ 
    \tilde{\lambda}  \in \mathrm{XR}(\Lambda), s\in \mathcal{S},	\label{eq:aggregatedBender_duplicated_feascut} 
\end{alignat}
\label{Benders_adaptive}
\end{subequations}
where $\mathcal{P}$ is a partition of $\mathcal{S}$ such that %for each $P\in\mathcal{P}$, 
%the 
there exist optimal dual solutions $\hat{\lambda}^s$ of subproblems $Q(x^*,\xi^s)$ 
%for each $s\in P$ 
{that} satisfy the following conditions:
\begin{subequations}
\begin{align}
    \left(\sum_{s\in P} p^s\right) \cdot \sum_{s\in P} p^s \left({h^s}^\intercal \hat{\lambda}^s\right) &= \left(\sum_{s\in P} p^s h^s\right)^\intercal \left(\sum_{s\in P} p^s \hat{\lambda}^s \right) && \text{for all } P \in \mathcal{P} \label{eq:condition1}\\
    \left(\sum_{s\in P} p^s\right) \cdot \sum_{s\in P} p^s \left({T^s}x^\intercal \hat{\lambda}^s\right) &= \left(\sum_{s\in P} p^s T^s x\right)^\intercal \left(\sum_{s\in P} p^s \hat{\lambda}^s \right) && \text{for all } P \in \mathcal{P}. \label{eq:condition2}
\end{align}\label{eq:conditions}
\end{subequations}
Then, $x^*$ is also an optimal solution of \eqref{eq:main}.
\end{prop}
\begin{proof}
For each $P\in\mathcal{P}$, since $\hat{\lambda}^{s}\in \mathrm{XP}(\Lambda)$ for all $s\in P$, then constraint \eqref{eq:aggregatedBender_duplicated} implies that
\[ p^P \cdot (h^P-T^P x)^\intercal \hat{\lambda}^{s'} \leq \sum_{s\in P} p^s\theta^s \qquad   s'\in P \]
and hence, by multiplying the above inequalities by $p^{s'}/p^P$ and summing them up over all $s' \in P$, we obtain: 
\[ p^P \cdot (h^P-T^P x)^\intercal \left(\frac{\sum_{s\in P} p^s \hat{\lambda}^s}{\sum_{s\in P} p^s}\right) \leq \sum_{s\in P} p^s\theta^s.  \]
The left-hand side term of this expression can be rewritten as
\[ \left(\sum_{s\in P} p^s\right) \left(\frac{\sum_{s\in P} p^s h^s - \sum_{s\in P} p^s T^s x}{\sum_{s\in P} p^s}\right)^\intercal \left(\frac{\sum_{s\in P} p^s \hat{\lambda}^s}{\sum_{s\in P} p^s} \right)  \]
and, if $\hat{\lambda}^{s}$ satisfies conditions~\eqref{eq:condition1} and~\eqref{eq:condition2}, then this term is equal to
\[\left(\sum_{s\in P} p^s\right) \left(\frac{1}{\sum_{s\in P} p^s} \sum_{s\in P} p^s \left({h^s - T^s x}\right)^\intercal \hat{\lambda}^s \right) %\leq \sum_{s\in P} p^s\theta^s 
\]
Thus, constraint \eqref{eq:aggregatedBender_duplicated} implies that
\[  \sum_{s\in P} p^s \left({h^s} - T^sx \right)^\intercal \hat{\lambda}^s \leq \sum_{s\in P} p^s\theta^s  \qquad  %\forall p\in P
P \in \mathcal{P}.\]
By aggregating these cuts for all $P\in \mathcal{P}$, this also implies that
\[ \sum_{P\in \mathcal{P}}\sum_{s\in P} p^s \left({h^s} - T^sx \right)^\intercal \hat{\lambda}^s \leq \sum_{P\in\mathcal{P}}\sum_{s\in P} p^s\theta^s, \]
which is equivalent to 
\[  \sum_{s\in S} p^s \left({h^s} - T^sx \right)^\intercal \hat{\lambda}^s  \leq \sum_{s\in S} p^s\theta^s := \Theta, \]
which are Benders cut for the single-cuts reformulation. Hence, $x^*$  is also an optimal solution of the single-cuts Benders reformulation.
{Since the problem \eqref{Benders_adaptive} is a relaxation of the original multi-cuts Benders reformulation \eqref{Benders_multicut}, and we just showed that an optimal solution $x^*$ of \eqref{Benders_adaptive} is also optimal for the problem \eqref{eq:main}, it follows that \eqref{Benders_adaptive} is a proper   reformulation of problem \eqref{eq:main}.} 
\end{proof}
For two partitions $\mathcal{P}^1, \mathcal{P}^2$ of $\mathcal{S}$, we call $\mathcal{P}^2$ \emph{a refinement} with respect to $\mathcal{P}^1$, denoted by $\mathcal{P}^1\preceq \mathcal{P}^2$, if and only if for any two subsets $P_r,P_q \subset \mathcal{S}$, $P_r,P_q \in \mathcal{P}^2$, there exists a set $P \in \mathcal{P}^1$ such that $P_r \cup P_q \subseteq P$. The following result follows from Proposition \ref{prop_adapt_valid_cut}: 
\begin{corollary}\label{cor:partitions}
Let $v(\mathcal{P})$ denote the optimal solution value of the problem \eqref{Benders_adaptive}. For a given family of partitions $\{\mathcal{P}^1, \dots, \mathcal{P}^t\}$ such that 
$\mathcal{P}^1 \preceq \dots \preceq \mathcal{P}^t \preceq \mathcal{S}$, 
we have
\[ v(\mathcal{P}^1) \leq \dots \leq v(\mathcal{P}^t) \leq v(\mathcal{S}). \]
\end{corollary}
Hence, by starting with a union of all scenarios and by iteratively refining this partition, Benders adaptive-cuts method provides a valid and non-decreasing lower bound, which eventually ensures the convergence of the method. The convergence follows from two aspects: first, from the convergence of the GAPM, and second from the convergence of the Benders method applied within each GAPM iteration. %{Maybe to move in the next section?}
In Section \ref{sec:implementation}, we discuss how to implement the Benders adaptive-cuts method based on an iterative refinement of partitions of $\mathcal{S}$ guided by the GAPM.

\subsection{Benders Adaptive-Single-Cuts} \label{sec:adaptive-single}
We notice that formulation \eqref{Benders_adaptive} includes one optimality cut for each element $P \in \mathcal{P}$, and thus, it can be interpreted as Benders multi-cuts method applied to the given partition $\mathcal{P}$. Similarly, one can consider its single-cuts counterpart, which is given as
\begin{subequations}
    \begin{alignat}{2}
 \min_{x\in \mathcal{X} ,\Theta_{\mathcal{P}} \in \mathbb{R}} ~&c^\intercal x + \Theta \\
\sum_{P\in \mathcal{P} } p^P (h^P-&T^P x)^\intercal \hat{\lambda}^{P} \leq  \Theta 
%\sum_{s\in P} p^s\theta^s  
&&\quad 
(\hat{\lambda}^{1}, \dots, \hat{\lambda}^{|\mathcal{P}|}) \in \mathrm{XP}(\Lambda)^{|\mathcal{P}|}  \label{eq:aggregatedBender_duplicated_single}\\
    (h^s-&T^s x)^\intercal \tilde\lambda  \leq 0 && \quad 
    \tilde{\lambda}  \in \mathrm{XR}(\Lambda), s\in \mathcal{S}.	\label{eq:aggregatedBender_duplicated_feascut_single} 
\end{alignat}
\label{Benders_adaptive_single}
\end{subequations}
In this model, we refer to constraints \eqref{eq:aggregatedBender_duplicated_single} as \emph{Benders adaptive-single-cuts}.
Using similar arguments as above, it is not difficult to see that conditions given in Proposition \ref{prop_optimal}   guarantee that the optimal solution $x^*$ of  \eqref{Benders_adaptive_single} is also optimal for the original problem \eqref{eq:main}.

It is worth mentioning that constraints 
\eqref{eq:aggregatedBender_duplicated_single} have been originally used by \cite{pay2017partition} under the name \emph{coarse cuts}. The authors generated these cuts in the early stage of their Benders decomposition scheme, when the initial solutions of the master problem are far from the optimal solution. The cuts have been used within a generation of additional valid inequalities (and not as a stand-alone problem formulation), with the aim of enhancing the convergence and improving the computational performance. 
Thus, with our results derived in Proposition \ref{prop_optimal}, we show that also \emph{coarse cuts}, when applied to a partitioning scheme guided by the GAPM, can lead to an alternative exact solution approach, to which we refer as Benders adaptive-single-cuts method.

\subsection{Implementation of Benders Adaptive-Cuts}\label{sec:implementation}

The result provided in Proposition~\ref{prop_optimal} indicates that it is sufficient to find a
%we need to find only a
partition of the scenarios, say $\mathcal{P}^*$, satisfying the conditions \eqref{eq:condition1} and \eqref{eq:condition2} 
%to find 
in order to correctly reformulate the problem~\eqref{eq:main} and find its  optimal solution.
%of the problem. 
Since 
%this partition 
$\mathcal{P}^*$ depends on the optimal solution $x^*$, it is not possible to construct it \emph{a priori}. However, we can start with  %partition represented by 
a single aggregated scenario $\mathcal{P}=\{\mathcal{S}\}$ and  iteratively refine the partitioning of $\mathcal{S}$  until the convergence criteria are met. That way, we are generating  new %aggregated
partition-based cuts, in a similar way to other Benders approaches, using a cutting plane procedure.  {At the same time, we will be potentially improving the incumbent solution.} An algorithmic implementation of this methodology is presented in Algorithm~\ref{alg:gen_benders}.

\begin{algorithm}[tbhp]%\label{alg:gen_benders}
\hspace*{\algorithmicindent} \textbf{Input:} Set of scenarios $\mathcal{S}$, with probabilities $p^s$ for $s\in\mathcal{S}$\\
 \hspace*{\algorithmicindent} \textbf{Output:} Optimal solution $\hat{x}$ and optimal value $z^\star$
	\begin{algorithmic}[1]
%		\REQUIRE {XX}
%	\ENSURE {XX}
	\State Set $t:=0$, $z_L^{(0)}:=-\infty$, $z_U:=\infty$, $z^\star:=\infty$ and $\mathcal{P}^{(0)} = \{\mathcal{S}\}$.
    %\State 
    \State \label{step:2} Let \textsc{MP} be the problem
\[         \min_{x\in\mathcal{X} ,\theta 
%\in \mathbb{R}^{|\mathcal{S}|}
\ge \LB
} c^\intercal x + \sum_{s \in \mathcal{S}} p^s \theta^s \]
%	\Loop
		\State \label{solve_MP}Solve \textsc{MP} 
		and let $(x^{(t)},{\theta}^{(t)})$ and $z_L^{(t)}$ be its optimal solution and its optimal value. 
		%\EM{Note that this problem can be unbounded. TO FIX \textbf{[REVIEWER 1, Point 5]}} 
		\label{alg:seq1_start}
		\IfThen{$z_U= z_L^{(t)}$}%{$z^\star=z_L^{(t)}$ }
		{\Return optimal solution $\hat x := x^{(t)}$ and optimal value $z^\star:=z_L^{(t)}$} \label{checkOptimality}
%         p^P \left[\sum_{k\in \mathbb{K}} \left({d_k^P}^\intercal \hat{\lambda}_k^P \right) + (h-Tx)^\intercal \hat{\mu}^P\right]  &\leq \sum_{s\in P} p^s\theta^s  & \forall P\in \mathcal{P}^{(t)}\\
% 		 \sum_{k\in \mathbb{K}} {d_k^s}^\intercal \tilde{\lambda}_k^s  + (h-Tx)^\intercal \tilde{\mu}^s &\leq 0, & \forall \big( \tilde{\lambda}^s,\tilde{\mu}^s \big) \in \mathcal{L}_s^{t},~\forall ~ s\in \mathcal{S}	
% \end{align*}
		\ForAll{$P\in \mathcal{P}^{(t)}$}
% 		    \begin{align*}
% \min_{y\geq 0} \sum_{k\in \mathbb{K}} q_{k}^\intercal y_k \\
% 	\mathcal{N} y_k &= d_k^P	& \forall k\in \mathbb{K}\\
% 	T x^{(t)} + W y &= h  
% \end{align*}
        \State Set $p^P := \sum_{s\in P} p^s$, $h^P := \frac{\sum_{s\in P} p^s h^s}{\sum_{s\in P} p^s}$,  $T^P := \frac{\sum_{s\in P} p^s T^s}{\sum_{s\in P} p^s}$ and $\xi^P := (T^P,h^P)$
		\If{subproblem $Q(x^{(t)},\xi^P)$ is infeasible}%and
		    %dual extreme ray $\tilde\lambda^P$\par
		    %\hskip\algorithmicindent of unboundedness violates $(h^P-T^P x^{(t)})^\intercal \tilde\lambda^P \leq 0$}  
		    \State Get a dual extreme ray $\tilde\lambda^P$ and add the following cut to \text{MP}  \Comment{feasibility cut} \label{alg:seq1_feascut}
             \[(h^P-T^P x)^\intercal \tilde\lambda^P \leq 0\]
% 		\ElsIf{$\hat{\lambda}^P$ are optimal dual solutions of subproblem $Q(x^{(t)},\xi^P)$ and \par
% 		\hskip\algorithmicindent $ p^P(h^P-T^P x^{(t)})^\intercal \hat{\lambda}^P  > \sum_{s\in P} p^s{\theta^{s}}^{(t)}$}
        \Else 
        \State Let $\hat{\lambda}^P$ be an optimal dual solution of $Q(x^{(t)},\xi^P)$ 
        \If{$ p^P(h^P-T^P x^{(t)})^\intercal \hat{\lambda}^P  > \sum_{s\in P} p^s{\theta^{s}}^{(t)}$}
		    \State Add to \text{MP} the following cut \Comment{optimality cut} \label{alg:seq1_optcut}
\[ p^P (h^P-T^P x)^\intercal \hat{\lambda}^P \leq \sum_{s\in P} p^s\theta^s \]
		   \EndIf
		   \EndIf
		\EndFor 
		\If{a feasibility or optimality cut has been added to \textsc{MP}}
		    \Goto{Step \ref{solve_MP}} \label{alg:seq1_end}
		\Else
		%\Comment{Aggregated problem is optimal. We revise the partition $\mathcal{P}^{(t)}$ for the given solution  $x^{(t)}$.}
%		
	    \ForAll{$s\in \mathcal{S}$}\label{alg:seq2_start}
	        \State Solve and store either extreme ray $\tilde{\lambda}^s$ % of unboundedness 
	        or optimal %extreme point 
	        solution $\hat{\lambda}^s$ of
	        %\par  	        \hskip\algorithmicindent as well as the optimal value of each subproblem
	        $Q(x^{(t)},\xi^s)$
	    \EndFor
	    \State Set $z_{U} := \min\left\{z_U,c^\intercal x^{(t)} + \sum_{s\in \mathcal{S}} p^s Q(x^{(t)},\xi^s)\right\}$
        \If{%there exists 
        $\exists$ $P\in \mathcal{P}^{(t)}$, $ s \in P$ s.t.   $Q(x^{(t)},\xi^s)$ is infeasible, or s.t.  $(\hat\lambda^s)_{s\in P}$ do not satisfy  %conditions
        \eqref{eq:conditions}} \label{alg:condition} %or \eqref{eq:condition2}}
            \State Run \texttt{refinement procedure} which
            refines $\mathcal{P}^{(t)}$ to obtain a new partition $\mathcal{P}^{(t+1)}$   \label{alg:refinement}
            \State Set $t := t+1$ and \Goto{Step \ref{checkOptimality}} \label{alg:disggregation_step}
        %\EndIf
        \label{alg:seq2_end}
        \Else %{$\mathcal{P}^{(t+1)} = \mathcal{P}^{(t)}$}
		    \State  \Return  $z^\star:=z_L^{(t)}, \hat x := x^{(t)}$
		 %\Else 
		 %\State $t := t+1$ and \Goto{Step \ref{solve_MP}} 
		\EndIf
	\EndIf
%	\EndLoop
	%\RETURN optimal solution $x^\star:=\bar{x}^{(t)}$, optimal partition $\mathcal{P}^\star:=\mathcal{P}^{(t)}$ and optimal value $z_L^{(t)}$
	\end{algorithmic}
	\caption{Implementation of the Benders Adaptive-cuts Method}
	\label{alg:gen_benders}
\end{algorithm} 

%The algorithm starts with a partition including all scenarios which is then refined  in the following iterations. 
In each iteration $t$, we denote by $\mathcal{P}^{(t)}$ the current partition of the scenario set $\mathcal{S}$. At the beginning, all the scenarios are aggregated, so we have $\mathcal{P}^{(0)}=\{\mathcal{S}\}$.
When solving the initial LP given in Step \ref{step:2}, we ensure that the problem is  bounded by restricting $\theta^s$ variables from below ($\theta \ge \LB$), where e.g., $\LB = \sum_{i=1}^n \min \{q_i \LB_i,q_i \UB_i\}$,  and $\LB_i$ and $\UB_i$ correspond to a global lower and upper bound of $y_i$, respectively. In each iteration $t$ we are  solving the associated problem relaxation given by \eqref{Benders_adaptive} with respect to $\mathcal{P}^{(t)}$. We then check if the conditions of Proposition \ref{prop_optimal} are met, and if yes, we terminate with a guarantee that an optimal solution is found. Otherwise, we refine the partition $\mathcal{P}^{(t)}$ and repeat the above procedure. We emphasize that we do not start solving \eqref{Benders_adaptive} with respect to $\mathcal{P}^{(t)}$ from scratch, but we keep all previously added cuts instead. Moreover, for any solution $x \in \mathcal{X}$ encountered along this process, we calculate the upper bound obtained by evaluating the expected value of the recourse, and potentially update the global upper bound. Thus, the algorithm either terminates because the condition of Proposition \ref{prop_optimal} is met, or because the lower bound at iteration $t$, corresponding to $v(\mathcal{P}^{(t)})$  (see Corollary \ref{cor:partitions}) matches the best found upper bound. 

At a more detailed level, in the first set of instructions (Steps \ref{alg:seq1_start}-\ref{alg:seq1_end}), we apply the usual Benders approach with the aggregated cuts, adding optimality and feasibility cuts at each iteration. Note that this step includes fewer optimality cuts (Step~\ref{alg:seq1_optcut}) than does the  Benders multi-cuts method  because we consider only each element of the partition, not each scenario. Similarly, we also add fewer feasibility cuts (Step~\ref{alg:seq1_feascut}) because the subproblem has to be feasible for the aggregated subproblems, not for each individual scenario.  Once the aggregated Benders problem is optimized, we revise each element of the partition $\mathcal{P}$ to verify that 
conditions \eqref{eq:condition1} and \eqref{eq:condition2} are satisfied for this subset of scenarios (Steps \ref{alg:seq2_start}-\ref{alg:seq2_end}) for the current %incumbent 
solution $x^{(t)}$. % (\texttt{disaggregation procedure})
Note that this is the only part of the algorithm where all second-stage subproblems are required to be optimized.  If there exists a set $P \in \mathcal{P}^{(t)}$ such that at least one of its subproblems $s \in P$ is infeasible, or if conditions of Proposition \ref{prop_optimal} are not met, then the set $P$ has to be refined to meet these conditions, and a new partition has to be created. The 
%procedure used to refine a subset of scenarios to satisfy conditions \eqref{eq:conditions} in the 
underlying \texttt{refinement procedure} (Step \ref{alg:refinement}) is problem specific, and we discuss a few examples in the next section. When the conditions stated in Step \ref{alg:condition} are satisfied, then $x^{(t)}$ is the optimal solution for problem~\eqref{eq:main}, according to Proposition~\ref{prop_optimal}. 
Note that if $\mathcal{P}=\mathcal{S}$, then the algorithm is equivalent to the Benders multi-cuts method.

A key part of the algorithm consists in refining of the current partition to satisfy conditions stated in Step \ref{alg:condition}. Let $P$ be an element from $\mathcal{P}^{(t)}$, we distinguish the following two cases:
\begin{itemize}
    \item When all scenarios $s \in P$ are feasible, one way to perform the refinement of $P$ is to group all scenarios with the same dual solution $\hat{\lambda}^s$, as presented in \citet{song2015adaptive}. However, this condition might be too strong, as shown in~\cite{ramirez2021generalized}, generating a partition with too many subsets. On the other hand, to compute the minimal partition $\mathcal{P}^{(t+1)}$ satisfying~\eqref{eq:conditions} for $x^{(t)}$ can be computationally difficult. 
Less restricting conditions can be obtained when the subproblem associated to the uncertainty parameters have particular substructures. One such example can be found in the stochastic flow problems, where many of the components of ${h}^{s}$ are equal to 0 for all scenarios. In this case, we only need to consider dual variables associated to the remaining components (see Section \ref{sec:SMCF} for more details).
\item If the current solution $x^{(t)}$ produces infeasible subproblems for some $s \in P$, we can group all these infeasible scenarios sharing the same dual extreme ray $\tilde{\lambda}$. In fact, since $(h^s-T^s x)^\intercal \tilde{\lambda} >0$ for these scenarios,  then the aggregated feasibility cut $p^P (h^P-T^P x)^\intercal \hat{\lambda} \leq 0$ will eliminate $x^{(t)}$ from the set of feasible solutions. Note that we only need to add one feasible cut for each subset of scenarios  sharing the same dual extreme ray. Indeed, this ensures that at the end of the algorithm all subproblems $s \in P$ are feasible, satisfying constraints~\eqref{eq:aggregatedBender_duplicated_feascut} from Proposition~\ref{prop_optimal}.  
%is the multi-commodity flow problem (cf.\ Section ??), where for each given scenario $s$, the condition ?? can be written as a sum over all commodities. In this case, imposing for each commodity that the scalar product of the expected dual values and the expected RHS is the same as the expected product of the two, implies condition ??.}     
\end{itemize}

%We remark that this algorithm can be seen as a case of the \emph{aggregated L-shaped} framework described in~\citet{biel2019dynamic}, so a formal proof of convergence and the maximum number of iterations required can be derived from these results. 

\section{Benders Adaptive-Cuts Applied to Stochastic Network Flow Problems}\label{sec:BendersStochNetworks}

Since the %middle of the past century,
1950's, network flow optimization problems have been comprehensively studied~\citep{costa2005survey,zargham2013accelerated,bertsimas2003robust,AhujaBook}. The problems are highly relevant to a variety of applications, including transportation~\citep{zapfel2002planning}, telecommunications~\citep{amaldi2008optimization,leitner2020exact,LjubicMPG21}, energy~\citep{jin2013joint,geidl2005modeling}, and others fields. Even though efficient combinatorial algorithms exist for specific problems such as max-flow or min-cost-flow~\citep{AhujaBook},
%have already been  solved efficiently, a long journey of 
for a plethora of related (sometimes NP-hard) problems, efficient  MIP-based exact solution methods need to be
developed. This is especially true when network flow problems  %specifically when similar graph structures 
appear in more complex optimization settings, such as mixed-integer linear programming, nonlinear programming, large-scale optimization and, in particular, stochastic programming. 

Stochastic network flows appear as a natural choice for many applications arising in transportation and logistics, where networks have to be designed subject to uncertain transportation demand. A similar situation occurs when a subset of facilities has to be open in order to serve customers' demands, but the actual values of this demand are only discovered in the future. 
%To showcase the application of the proposed Benders  adaptive-cuts methodology, 
In our study, we focus on three stochastic network flow problems: a) the capacity planning problem (CPP), b) the stochastic multicommodity flow problem (SMCF), and c) the stochastic facility location problem with risk aversion (FL-CVaR).
%To provide more detailed examples of how to apply ,  where the Benders method has been studied and applied in the literature. 
%Specifically, we study two such problems: 
%and  in the next section, we compare the computational performance of three different Benders implementations: single-cut, multi-cut and adaptive-cut.
We choose the first two problems because they also
frequently appear in the related literature when it comes to efficient implementations of the Benders decomposition approach \citep{rahmaniani2018accelerating,crainic2021partial}. The third problem complements the former two, as it involves binary first-stage variables and models a risk-averse objective function. 
Moreover, Benders single-cuts and multi-cuts methods for these three problems behave differently, as we will see in the computational results.
 
All three problems 
are defined on a directed graph $G=(V,E)$, where in the first-stage a decision is made  %is made by deciding either 
concerning the structure of the network and the capacity of some of its nodes or arcs. In the second-stage, a routing decision is made once the uncertainty (e.g., random demand) is revealed. In both cases, the cost of the first-stage decisions plus the expected routing cost of the second-stage, are minimized. 
In the remainder of this section, for each of the three problems, we provide
%We now present the 
a mathematical formulation  and a specific recipe for the refinement procedure. Results of our implementation and a comparison with alternative methods 
%of the underlying cutting plane procedure are  presented %in our computational study 
are given  in Section~\ref{sec:computational}.
%used to assess these problems in the Benders framework from Section~\ref{sec:statement}.

\subsection{The Capacity Planning Problem (CPP)}

Our first example is a %classical 
stochastic network flow problem, originally proposed by~\cite{louveaux1988optimal}, in which we are given  a bipartite graph $G=(V,E)$, representing an electricity planning network with the set of nodes $V=V_L\cup V_R$ and the set of arcs $E=V_L\times V_R$.
The nodes from $V_L$ represent power terminals, and the nodes from $V_R$ are the customer nodes with uncertain demand. There are $|\mathcal{K}|$ different resources, and 
a unit of capacity allocated to terminal
$i \in V_L$ uses $a_{ik}$ units of resource $k \in \mathcal{K}$. Total availability of each resource $k \in \mathcal{K}$ limited to $r_k >0$.
In the first stage one has to decide about the capacities of each power terminal $i \in V_L$, while respecting the resource limitations. 
In the second stage, once the demands at the nodes $j \in V_R$ are revealed, the electricity  can  be transported from the terminals to demand nodes. %, with an associated transportation cost/profit. 
Transporting one unit of demand between nodes $i \in V_L$ and $j \in V_R$ invokes the cost/profit of $e_{ij}$, which can include the transportation cost and revenue from sales, so $e_{ij}$ can be positive (cost) or negative (profit). Note that the demand does not need to be satisfied, so installing zero capacity at the terminals is a feasible solution for any scenario.  The objective function consists of  minimizing the capacity installation cost at the first stage plus the expected transportation cost/profit  at the second stage.  
%the uncertain demand nodes ($V_R$) from the supply nodes ($V_L$) at minimum cost, including the capacity installation cost on the supply nodes and the transportation cost over arcs.
%$x\in \mathbb{R}_+^{|V_L|}$ represents the supply at nodes $i$ with a unit cost $c_i$. Model with the same structure was originally posed by~\cite{louveaux1988optimal}. 
%\todo[inline]{Even though the problem has been solved with high confidence (95\%) using approach in~\cite{powell2004learning}, our purpose is to compare Adaptive-cut version of Benders against predecessor implementations Multi and Single-cut, rather than pursuing better performance than existent efficient algorithms for both this and Stochastic Multicommodity Flow Problem.}

\paragraph{Mathematical formulation:}
Following our notation, the CPP can be stated as follows:
\begin{subequations}
    \label{CP_stoch}
    \begin{align}
        \min_{x\in\mathbb{R}_+^{|V_L|}} \sum_{i\in V_L} c_i x_i &+ \sum_{s\in\mathcal{S}} p^s Q(x,\xi^s)& \\
        \sum_{i\in V_L} a_{ik} x_i &\leq r_k &  ~k\in\mathcal{K} \label{CP_stoch:eq_budget}
    \end{align}
\end{subequations}
In this model, the first-stage variable $x\in\mathbb{R}_+^{|V_L|}$ represents the capacity installed at each supply node $i\in V_L$ with unit cost $c_i$. The budget constraints \eqref{CP_stoch:eq_budget} are imposed for each resource $k\in \mathcal{K}$. 
%where $a_{ik}$ is consumption of recourse $k\in\mathcal{K}$ when allocating one unit of production in plant $i\in V_L$. 
Given $\hat{x}$, the second-stage subproblem $Q(\hat{x},\xi^s)$ is a min-cost flow problem
\begin{subequations}
    \label{CP_stoch_subp}
    \begin{align}
        Q(\hat{x},\xi^s) := \min_{y^s\in\mathbb{R}_+^{|E|}  } &~  \sum_{ij\in E} e_{ij}  y^s_{ij}&   \\
        \sum_{i\in V_L} y^s_{ij} & \leq d^s_j, &  ~j\in V_R
        \label{eq:CPP_dem}\\
        \sum_{j\in V_R} y^s_{ij} &\leq  \hat{x}_i &  ~i \in V_L 
        \label{eq:CPP_cap}
    \end{align}
\end{subequations}
where the second-stage variable $y^s\in\mathbb{R}_+^{|E|}$ represents the flow from  supply nodes to  demand nodes.  This flow %has an associated unit cost $e_{ij}$, and it 
must satisfy the capacity constraint \eqref{eq:CPP_cap} provided by first-stage variables $\hat{x}$ at each supply node and the maximum (random) demand $d^s$ at each demand node, cf. constraints \eqref{eq:CPP_dem}. Note that $e_{ij}$ can be negative, representing a profit for assigning flow on arc $ij\in E$.
Given the first-stage decision $\hat{x}$, the dual of the subproblem associated to scenario $s \in \mathcal{S}$ is as follows:
\begin{subequations}
    \label{cp_dualsp}
    \begin{align}
        \mathcal{Q}(\hat{x},\xi^s):= -\min_{\mu^s \in \mathbb{R}_+^{|V_R|}, \nu^s\in \mathbb{R}_+^{|V_L|}} \sum_{j\in V_R} d^s_j \mu^s_j &+ \sum_{i\in V_L} \nu^s_{i} \hat{x}_i& \\
        \mu^s_j + \nu^s_{i} &\geq -e_{ij} & ~ ij\in E %\\
%         \mu^s_j, \nu^s_i \geq 0 & \forall~i\in V_L,j\in V_R
    \end{align}
\end{subequations}

\paragraph{Benders cuts:}
An important remark about the capacity planning problem is that the subproblem solution $y^s=0$ is feasible for any feasible first-stage decision $\hat{x}$, so there always exists an optimal solution for dual of the subproblem  $\mathcal{Q}(\hat{x},\xi^s)$. Hence, there are only Benders optimality cuts to be dealt with when implementing a Benders decomposition procedure.

Given a first-stage solution $\hat{x}$  and the corresponding optimal  solutions $\left(\hat{\mu}^s,\hat{\nu}^s\right)$ of the dual of $Q(\hat{x},\xi^s)$ for each $s\in \mathcal{S}$, %denoted by $(\hat\mu^s,\hat\nu^s)_{s\in\mathcal{S}}$ for short, 
we present the inequalities that are added to reformulations \eqref{Benders_multicut} and \eqref{Benders_singlecut} for the cases of the multi-cuts and single-cuts implementation, respectively:
%EDITOR: Please ensure that the intended meaning has been maintained in the above edits.

\begin{equation*}
    -\sum_{j\in V_R} d_j^s \hat{\mu}^s_j - \sum_{i\in V_L} \hat{\nu}_i^s x_i  \le \theta^s \quad   ~ s\in \mathcal{S} %, (\hat{\mu},\hat{\nu})\in \mathcal{L}_s^t
    %\tag{Optimality multi-cuts}
    \tag{{CPP multi-cuts}}
\end{equation*}

\begin{equation*}
    -\sum_{s \in \mathcal{S}} p^s\left[\sum_{j\in V_R} d_j^s \hat{\mu}^s_j + \sum_{i\in V_L} \hat{\nu}_i^s x_i\right]   \le \Theta. %\quad \forall~ (\hat{\mu},\hat{\nu})\in \mathcal{L}_s^t
    %\tag{Optimality single-cut}
    \tag{{CPP single-cuts}}
\end{equation*}
 
Given a partition  $\mathcal{P}$ of the set $ \mathcal{S}$, for any $P\in \mathcal{P}$, {let $d_j^P:=\sum_{s\in P}p^s d_j^s/p^P$ be the expected demand of the node $j \in V_R$ with respect to the aggregated scenarios from $P$. }
Let $(\hat\mu^P,\hat\nu^P)$ be an optimal solution of  the subproblem~\eqref{cp_dualsp} using  demands $d^P$.
It follows from Proposition~\ref{prop_adapt_valid_cut} that 
the following cuts are valid for the CPP:  %when reformulating \eqref{CP_stoch}, as in \eqref{Benders_multicut} % and \ref{prop_optimal}, so that enforcing equality \eqref{eq:condition_CP}, 
\begin{equation*}
    -p^P\left[\sum_{j\in V_R} d_j^P \hat{\mu}^P_j + \sum_{i\in V_L} \hat{\nu}_i^P x_i\right]   \le \sum_{s\in P} p^s \theta^s \quad ~ P \in \mathcal{P}. %(\hat{\mu},\hat{\nu})\in \mathcal{L}_s^t
    \label{eq:adaptiveCPP}
    %\tag{Optimality adaptive-cuts}
    \tag{{CPP adaptive-cuts}}
\end{equation*}

\paragraph{Refinement procedure:}
In order to explain the refinement procedure for the CPP which we have implemented as step \ref{alg:disggregation_step} of Algorithm~\ref{alg:gen_benders}, we start with the following result derived from Proposition~\ref{prop_optimal}:
\begin{corollary}
Given a partition $\mathcal{P}$ of the set $\mathcal{S}$, the following condition ensures that adding %Benders %optimality 
%adaptive-cuts 
\eqref{eq:adaptiveCPP} 
for each $P\in\mathcal{P}$ provides an optimal solution of
the CPP:
%Problem~\eqref{CP_stoch}. 
\begin{equation}
        \sum_{j\in V_R} \left[ \left(\sum_{s\in P} \frac{p^s}{p^P} d^s_j \right)\cdot \left(\sum_{s\in P} \frac{p^s}{p^P} \hat{\mu}^s_j \right)\right] = \sum_{j\in V_R} \frac{1}{p^P} \sum_{s \in P} p^s d_j^s \hat{\mu}^s_j, \qquad \text{for all $P \in \mathcal{P}$}
        \label{eq:condition_CP}
\end{equation}
\end{corollary}

%From Theorem [1] in Ramirez-Pico \& Moreno, we know there exists a finite partition $\mathcal{P}$ for problem \eqref{CP_stoch} such that the optimal solution of the original problem, including the whole set of scenarios, is equal to a reduced problem with elements $P\in\mathcal{P}$ taking the role of `new scenarios', such that a new (probably smaller) problem has the same optimal solution than the original one. For the capacity planning the necessary condition fr such equivalence can be written down as follows

%As in our previous example, a natural case where this condition is satisfied is when either $d^s = d^{s'}$ or $\hat\mu^s = \hat\mu^{s'}$ for all pairs of scenarios $s,s'\in P$.

Sufficient conditions for a given partition $\mathcal{P}$ to satisfy the above equation are, for example, when either $d^s = d^{s'}$ or $\hat\mu^s = \hat\mu^{s'}$ for all pairs of scenarios $s,s'\in P$.  Indeed, for the vector of dual variables ($\mu$) and the vector of the associated right-hand-side ($d^s$), equation \eqref{eq:condition_CP} requires that the weighted average (for a given $P$) of their scalar product is equal to the scalar product of their weighted averages. 
%multiplying and then taking the weighted average, is equal to take the average independently and then multiply them.
This condition is naturally satisfied when one of these vectors is a constant ($d$, $\mu$, or both) across all scenarios $s \in P$. 
Hence, for each $P \in \mathcal{P}$, we construct $\mathcal{P}'$ by solving the subproblem~\eqref{cp_dualsp} for each $s\in P$ and group them  into subsets where $\mu^s$ has the same value.

\subsection{The Stochastic Multicommodity Flow Problem (SMCF)}
\label{sec:SMCF}

Our second problem is %from 
a network design optimization problem under uncertainty. %theory, wherein 
We are given a set of commodities $\mathbb{K}$, such that for each $k \in \mathbb{K}$ its origin $O(k) \in V$ and its destination $D(k)\in V$, are known, but its demand $d_k$ is subject to uncertainty. We are also given the cost of routing a single unit of demand of $k$ along an arc $ij \in E$, denoted by $c_{ijk} >0$. In the first stage, we need to decide on the fraction of the given capacity $u_{ij}>0$ of each arc $ij \in E$ that need to be installed in order to simultaneously  route all the commodities. Once the uncertainties are revealed, the 
%first-stage decision is the capacity of the arcs in the network, while the 
second-stage decision consists of finding the minimum-cost routing of the commodities over the network installed in the first stage.  
%routing the  is to solve a multicommodity flow problem satisfying the uncertain demand over the network constructed in the first stage. 
The problem has been widely studied both in its  integer~\citep{gendron1999multicommodity} and linear~\citep{Barnhart2001} versions, that is, where the capacity of each arc is either 0 or a nominal capacity $u_{ij}$ for the former, or a fraction of this nominal capacity, for the latter.
Various Benders approaches have been developed to solve the problem; the most recent and state-of-the-art implementation of Benders decomposition by \cite{rahmaniani2018accelerating} can efficiently address instances of the SMCF  with a small number of scenarios (between 16 to 64). 
{However, our aim is to focus on more challenging SMCF instances for which  the size of the uncertainty set can be very large (and could go up to 50,000 scenarios, as we shall see in our numerical study). Our goal is to understand the impact of the proposed Benders adaptive-cuts and to obtain possible improvements of the state-of-the-art for  solving these challenging instances.}

\paragraph{Mathematical formulation:}
The SMCF is formulated as follows
\begin{equation}
\label{shortsub}
\min\limits_{x\in\left[0,1\right]^{|E|} } \sum_{ij\in E} f_{ij}x_{ij} + \sum_{s \in \mathcal{S}} p^s Q(x,\xi^s),
\end{equation}
where the first-stage variable $0 \le x_{ij} \le 1$ indicates the fraction of the nominal capacity $u_{ij}$ that should  be installed on the arc $ij \in E$, and $f_{ij} >0$ is the installation cost per unit of capacity. We are given a discrete set of scenarios $\mathcal{S}$ modeling the uncertain realizations of the demand of each commodity $k \in \mathbb{K}$.
For a given first-stage decision $\hat{x}$, the second-stage problem $Q(\hat{x},\xi^s)$ associated to scenario $s \in \mathcal{S}$, is a multicommodity flow problem on the network $G$ with arc capacities defined as $\hat x_{ij} u_{ij}$ and with demands $d^s_k \ge 0$, for each $k \in \mathbb{K}$:
\begin{subequations}
	\begin{align}
	Q(\hat{x},\xi^s) :=&\min \sum_{ij\in E}\sum_{k\in \mathbb{K}} c_{ijk}y_{ijk}^s  \label{of1}\\
	\sum_{j:ij\in E} y_{ijk}^s-\sum_{j:ji\in E} y_{jik}^s & =
	\begin{cases}
	d_k^s& \text{ if } i=O(k) \qquad \\
	-d_k^s& \text{ if }i=D(k) \qquad \\
	0& \text{ otherwise }
	\end{cases}
	& ~ i\in V, k \in \mathbb{K}   \label{FlowConst}\\
	\sum_{k\in \mathbb{K}} y_{ijk}^s &\le u_{ij} \hat{x}_{ij} & ~ ij \in E   \label{CapacityConst}\\
	y_{ijk}^s &\ge 0 &  ~ij\in E, k \in \mathbb{K}
	\end{align}
	\label{MainpSFCMFP}
\end{subequations}
In this model, the second-stage variables $y_{jik}^s \ge 0$ indicate the amount of flow of commodity $k$ routed through the arc $ij \in E$ in scenario $s$.
%That is, the first-stage decisions $\hat{x}$ define the fraction of the nominal capacity $u_{ij}$ for each arc $ij\in E$, which has a cost $f_{ij}$. In the second-stage problem, each commodity $k\in \mathbb{K}$ must be routed over the network with a stochastic demand $d_k^s$ for each scenario $s\in\mathcal{S}$, satisfying 
Constraints \eqref{FlowConst} are flow-preservation constraints, whereas the capacity constraints \eqref{CapacityConst} guarantee that the installed capacity on each arc cannot be exceeded. 
%all commodities routed through each arc, with a unitary cost $c_{ijk}$ for each unit of commodity $k$ using arc $ij\in E$.

To construct Benders cuts for this problem, we define the dual form of \eqref{MainpSFCMFP} for a given first-stage solution $\hat{x}$.
\begin{subequations}
	\begin{align}
    	Q^D(\hat{x},\xi^s):=  \max\limits_{\lambda^s \in \mathbb{R}^{|V||\mathbb{K}|}, \mu^s \in \mathbb{R}^{|E|}_{+}} \sum_{k\in \mathbb{K}} d_k^s (\lambda_{O(k),k}^s - \lambda_{D(k),k}^s) - \sum_{ij\in E} \left( u_{ij} \hat{x}_{ij} \right) \mu_{ij}^s
    	&   \label{objDualMainSSP}\\
    	\lambda_{ik}^s-\lambda_{jk}^s-\mu_{ij}^s  \le  c_{ijk} \quad ij\in E,~k\in \mathbb{K}.   \label{flowdualSSP} %\\
    	%\lambda_{ik}^s&\in \mathbb{R}&\forall~ i\in V,~k\in\mathbb{K}\\
    %	\mu_{ij}^s &\ge 0 &\forall~ ij\in E   \label{nat_const_dualSSP}
	\end{align}
	\label{DualMainSSP}
\end{subequations}
Notably, for this problem, $\hat{x}$ does not always yield a feasible subproblem \eqref{MainpSFCMFP}, so its dual \eqref{DualMainSSP} could be unbounded in which case we will need to derive 
%suitable to obtain a 
feasibility cuts from %one of 
its extreme rays.

\paragraph{Benders cuts:}
%From Benders presentation in section \S xx, we know there are two valid formulations \eqref{Benders_multicut} and \eqref{Benders_singlecut} for the proposed problem \eqref{MainpSFCMFP}, such that the second stage value is replaced by either as many variables as scenarios exists $\theta_s ~\forall s\in\mathcal{S}$ or a single variable $\Theta$, respectively.
From \eqref{DualMainSSP}, we define the feasibility cuts, which are the same no matter if the single-cuts or the multi-cuts version is utilized. %Let $\Lambda$ be the feasible space of $Q^D(x,\xi^s)$. Hence, at each iteration, any new violated cut is added to master problem \eqref{Benders_multicut} or \eqref{Benders_singlecut} in the following manner
For a given $\hat{x}$, where $Q^D(\hat{x},\xi^s)$ is unbounded for some scenario $s\in\mathcal{S}$,  let $(\tilde\lambda^s, \tilde\mu^s)$ be an extreme ray of the corresponding subproblem. Hence, the following feasibility cuts are added to master problem \eqref{Benders_multicut} or \eqref{Benders_singlecut}: 

%\paragraph{Feasibility cut (single-cut and multi-cut):}
\begin{equation*}
    \sum\limits_{k\in \mathbb{K}} d_k^s \left(\tilde{\lambda}_{O(k),k}^s - \tilde{\lambda}_{D(k),k}^s\right) - \sum_{ij\in E} u_{ij} \tilde{\mu}_{ij}^s x_{ij}  \le 0 %\quad ~\forall ~ s\in \mathcal{S}, (\hat{\lambda},\hat{\mu})\in \mathrm{XR}(\Lambda)
    \tag{{Feasibility cuts}} \label{Feasibility cuts}
\end{equation*}

Similarly, in the case that $Q^D(\hat{x},\xi^s)$ is feasible, let $(\hat\lambda^s, \hat\mu^s)$ be the optimal solution of $Q^D(\hat{x},\xi^s)$ for each $s\in\mathcal{S}$; then, the following optimality cuts are added to the master problem, if violated:
\begin{equation*}
    \sum_{k\in \mathbb{K}} d_k^s \left(\hat{\lambda}_{O(k),k}^s - \hat{\lambda}_{D(k),k}^s\right) - \sum_{ij\in E} u_{ij} \hat{\mu}_{ij}^s x_{ij}  \le \theta^s  \quad s \in \mathcal{S} \tag{Optimality multi-cuts} \label{Optimality multi-cuts}
\end{equation*}

\begin{equation*}
    \sum_{s \in \mathcal{S}} p^s\left[ \sum_{k\in \mathbb{K}} d_k^s \left(\hat{\lambda}_{O(k),k}^s - \hat{\lambda}_{D(k),k}^s\right) - \sum_{ij\in E} u_{ij} \hat{\mu}_{ij}^s x_{ij}\right]   \le \Theta \tag{Optimality single-cuts} \label{Optimality single-cuts}
\end{equation*}

Whilst the Benders multi-cuts method may add both \eqref{Feasibility cuts} and \eqref{Optimality multi-cuts} in a single iteration, the single-cuts method can generate a new optimality cut only when all subproblems have an optimal solution.

%\paragraph{Adaptive-cut Framework:}
On the other hand, following Proposition~\ref{prop_adapt_valid_cut}, for a given partition $\mathcal{P}$ of $\mathcal{S}$, the feasibility and optimality aggregated cuts that can be added to the master problem, if violated, are

\begin{equation*}
\label{feasaggcuts}
p^P \left[ \sum_{k\in \mathbb{K}} d_k^P \left(\tilde{\lambda}_{O(k),k}^P - \tilde{\lambda}_{D(k),k}^P\right) - \sum_{ij\in E} u_{ij} \tilde{\mu}_{ij}^P x_{ij}\right]   \le 0 \quad  ~ P\in \mathcal{P}
%, (\hat{\lambda},\hat{\mu})\in \mathcal{E}_P^t
\tag{Feasibility adaptive-cuts}
\end{equation*}

\begin{equation*}
\label{optaggcuts}
p^P \left[ \sum_{k\in \mathbb{K}} d_k^P \left(\hat{\lambda}_{O(k),k}^P - \hat{\lambda}_{D(k),k}^P\right) - \sum_{ij\in E} u_{ij} \hat{\mu}_{ij}^P x_{ij}\right]   \le \sum_{s\in P} p^s \theta^s \quad  ~ P\in \mathcal{P}
%, (\hat{\lambda},\hat{\mu})\in \mathcal{L}_P^t,
\tag{Optimality adaptive-cuts}
\end{equation*}
where $(\tilde{\lambda}^P,\tilde{\mu}^P)$ and $(\hat{\lambda}^P,\hat{\mu}^P)$ are either extreme rays or optimal extreme points of the  subproblem \eqref{DualMainSSP} considering the aggregated demand $d_k^P:=\sum_{s\in P}p^s d_k^s/p^P$ for each $P\in\mathcal{P}$. 

Note that feasibility cuts are required to obtain a feasible solution for the weighted aggregated demand $d^P$ for each $P\in\mathcal{P}$, not for each scenario $s\in \mathcal{S}$, which  potentially reduces the number of feasibility cuts required to obtain a feasible solution for the problem. 

\paragraph{Refinement procedure:}
The following corollary follows directly from Proposition~\ref{prop_optimal}: 
\begin{corollary}
Given a partition $\mathcal{P}$ of the set $\mathcal{S}$, the following condition guarantees that adding Benders adaptive-cuts for each $P\in\mathcal{P}$ is sufficient to obtain an optimal solution of the SMCF: 
\begin{multline}
        \sum_{k\in\mathcal{K}} \left[ \left(\sum_{s\in P} \frac{p^s}{p^P} d^s_k\right)\cdot \left(  \sum_{s\in P}\frac{p^s}{p^P}\left(\lambda^s_{O(k),k}-\lambda^s_{D(k),k}\right) \right)\right] \\ = \sum_{k\in\mathcal{K}} \left[ \sum_{s\in P} \frac{p^s}{p^P} \cdot d^s_k\cdot \left(\lambda^s_{O(k),k}-\lambda^s_{D(k),k}\right) \right]
      \label{eq:SFCMCF_condition}
\end{multline}
\end{corollary}

Indeed, this condition follows from \eqref{eq:condition1} because only the demand is uncertain. Note that a given element $P \in \mathcal{P}$ satisfies condition \eqref{eq:SFCMCF_condition}, for example, when either 
\begin{align}
d_k^s = d_k^{s'}  \text{ or } \lambda^s_{O(k),k}-\lambda^s_{D(k),k} = \lambda^{s'}_{O(k),k}-\lambda^{s'}_{D(k),k} \quad k\in\mathcal{K}
\end{align}
for all pairs of scenarios $s,s'\in P$. 
%Hence, we can use this rule to refine the partition $\mathcal{P}$ in each iteration of Algorithm~\ref{alg:gen_benders}.
Hence, we apply the following rule to create $\mathcal{P}'$, a refinement of a given partition $\mathcal{P}$:
For each $P\in \mathcal{P}$, we solve the dual subproblem \eqref{DualMainSSP} for each $s\in P$. If there are unbounded subproblems, we group the scenarios sharing the same extreme ray. For the feasible subproblems, we group them into subsets where the vector $(\lambda^s_{O(k),k}-\lambda^s_{D(k),k})_{k\in\mathcal{K}}$ has the same value.
%See~\cite{ramirez2021generalized} for more details. 

\subsection{The Stochastic Facility Location Problem under risk aversion (FL-CVaR)}
\label{sec:FLCVAR}

Our third problem is a facility location problem 
%under 
with uncertain demand and a risk-averse decision maker. We are given a set of potential locations $I$ to install a facility, each of them with a fixed installation cost $f_i \ge 0$ and a capacity $K_i \ge 0$ for $i\in I$. 
For each client $j\in J$, and each facility $i \in I$, a transportation cost  $c_{ij}\ge 0$ is incurred for each unit of demand of client $j$ if $j$ assigned to facility $i \in I$.
The demand $d_j \ge 0$ of each client $j\in J$ is uncertain. 
 In the first stage, the decision maker decides which facilities to open. In the second stage, upon the realization of the uncertain demand (say $\tilde{d}_j$), each client is allocated to one of the open facilities, so that their capacities are respected.  
The problem is to minimize the total cost, involving the facility installation cost (first stage) and the expected value of the transportation cost of clients to facilities (second stage).
%To complex the problem, 
We study a risk-adverse version of this problem, replacing the expected value of the second stage with a risk-measure known as the Conditional Value-at-Risk (CVaR). Given a level of risk-aversion $\sigma$, we search for an optimal subset of facilities to install so that their opening cost plus the expected value of the worst $\sigma$-quantile of the random distribution of the transportation cost is minimized. 

\paragraph{Mathematical formulation:} The FL-CVaR can be formulated as follows:
\begin{subequations}
\begin{align}
  \min_{x\in\{0,1\}^{|I|}}  \sum_{i\in I} f_i x_i &+ \textrm{CVaR}_\sigma\left[Q(x,\xi)\right] \label{eq:FL_objOrig}\\
    \sum_{i\in I} K_i x_{i} &\geq D \label{eq:LF_capacOrig}
\end{align}
\end{subequations}
where binary first-stage variables $x_i$ indicate if the facility $i\in I$ is installed or not. The only constraint to the problem ensures that there must be enough capacity installed to satisfy all the demand, where $D:=\max_{\xi\in\Omega} \{\sum_{i\in J} {d}^\xi_j\}$. Given a set of installed facilities $\hat{x}$, the second-stage problem minimizes the transportation cost between the clients and the facilities:
\begin{subequations}
\begin{align}
   Q(\hat{x},\xi):=  \min_{y\in \mathbb{R}_+^{|I|\times |J|}} &\ \sum_{i\in I} \sum_{j\in J} c_{ij} y_{ij} &\\
   \sum_{i\in I} y_{ij} &\geq d^\xi_{j} &  j\in J \label{eq:FL_demand}\\
   \sum_{j\in J} y_{ij} &\leq K_i \cdot \hat{x}_i &  i\in I. \label{eq:FL_capac}
\end{align}
\end{subequations}
Here, the second-stage variables $y_{ij}$ represent the units of demand of client $j\in J$ assigned to facility $i\in I$. Constraint~\eqref{eq:FL_demand} indicates that all the demand of client $j$ must be satisfied, and constraint~\eqref{eq:FL_capac} enforces that the demand assigned to a given facility $i\in I$ does not exceed its capacity $K_i$ (if the facility $i$ is installed) or zero (if the facility is not installed). In case of a discrete set of scenarios $\mathcal{S}$, the risk-aversion measure CVaR with a given risk-level $\sigma$ can be formulated as a linear program, see \citet{rockafellar2000optimization}. That is, we can replace the objective function of the  problem~\eqref{eq:FL_objOrig} by
\begin{equation}
      \min_{x\in\{0,1\}^{|I|}, \tau\in\mathbb{R}}  \sum_{i\in I} f_i x_i + \tau + \frac{1}{1-\sigma}\sum_{s\in\mathcal{S}} p^s Q'(x,\tau,\xi^s)
\end{equation}
where the modified second-stage problem for a given $s\in\mathcal{S}$ reads as follows:
\begin{subequations}
\label{subeq:FL}
\begin{align}
   Q'(\hat{x},\hat{\tau},\xi^s):=  \min_{y\in \mathbb{R}_+^{|I|\times |J|}, z\geq 0} z &&&\\
   \sum_{i\in I} \sum_{j\in J} c_{ij} y_{ij} &\leq z + \hat{\tau} &\\
   \sum_{i\in I} y_{ij} &\geq d^s_{j} &  j\in J \\
   \sum_{j\in J} y_{ij} &\leq K_i \cdot \hat{x}_i &  i\in I.
\end{align}
\end{subequations}
The dual of this subproblem is as follows:
\begin{subequations}
\begin{align}
    Q'(\hat{x},\hat{\tau},\xi^s):=  \max_{\beta\in\mathbb{R}_+^{|J|},\gamma\in\mathbb{R}_+^{|I|}}\ -\alpha \hat\tau  + \sum_{j\in J} \beta_j d^s_j &- \sum_{i\in I} \gamma_i K_i \hat{x}_i \\
    -\alpha c_{ij}  + \beta_j - \gamma_i &\leq0 & i\in I, j\in J \label{eq:dual}\\
    0\leq \alpha &\leq 1.&&&
\end{align} \label{eq:FLCVAR_dualsubproblem}
\end{subequations}

\paragraph{Benders cuts:} Note that, given~\eqref{eq:LF_capacOrig}, the subproblem \eqref{subeq:FL} is always feasible, for any $s\in\mathcal{S}$ and any feasible first-stage decision $\hat{x}$, so only Benders optimality cuts need to be constructed.

Given a first-stage solution ($\hat{x}$, $\hat{\tau}$) and a corresponding dual solution $(\alpha^s,\beta^s,\gamma^s)$ of $Q'(\hat{x},\hat{\tau},\xi^s)$ for each scenario $s\in\mathcal{S}$, the optimality cuts to be added to master problem \eqref{Benders_multicut} or \eqref{Benders_singlecut} are:
\begin{align}
     -\alpha^s \hat\tau  + \sum_{j\in J} \beta^s_j d^s_j - \sum_{i\in I} \gamma^s_i K_i \hat{x}_i \leq \theta^s  \quad s\in\mathcal{S}  \tag{FL-CVaR multi-cuts} \\
     -\left(\sum_{s\in\mathcal{S}} p^s \alpha^s\right) \hat\tau  + \sum_{j\in J}\sum_{s\in\mathcal{S}} p^s  \beta_j^s d^s_j - \sum_{i\in I} \left(\sum_{s\in\mathcal{S}} p^s \gamma^s_i\right) K_i \hat{x}_i \leq \Theta.  \tag{FL-CVaR single-cuts}
\end{align}

Given a partition $\mathcal{P}$ of $\mathcal{S}$, let $d^P_j := \sum_{s\in P} p^s d_j^s / p^P$ be the expected aggregated demand of customer $j\in J$ and let $(\alpha^P,\beta^P,\gamma^P)$ be an optimal dual solution of the subproblem with demand $d^P$. Hence, by Proposition~\ref{prop_adapt_valid_cut} the following cuts are valid for the FL-CVaR problem:
\begin{align}
     -p^P \left[\alpha^P \hat\tau  + \sum_{j\in J} \beta^P_j d^P_j - \sum_{i\in I} \gamma^P_i K_i \hat{x}_i \right] \leq \sum_{s\in P} p^s \theta^s  \quad P \in \mathcal{P}.  \tag{FL-CVaR adaptive-cuts}
\end{align}

\paragraph{Refinement procedure:} As before, we provide a corollary of Proposition~\ref{prop_optimal} for this problem, which guide us on how to apply a proper refinement of $\mathcal{P}$ for the FL-CVaR problem:
\begin{corollary}
Given a partition $\mathcal{P}$ of the set $\mathcal{S}$, the following condition guarantees that adding Benders adaptive-cuts for each $P\in\mathcal{P}$ is sufficient to obtain an optimal solution of the FL-CVaR problem: 
\begin{equation}
\sum_{j\in J} \left[ \left( \sum_{s\in P} \frac{p^s}{p^P} d_j^s\right) \cdot  \left( \sum_{s\in P} \frac{p^s}{p^P} \beta_j^s\right)\right] = \sum_{j\in J}  \sum_{s\in P} \frac{p^s}{p^P} d_j^s \beta_j^s, \qquad \text{for all }p\in\mathcal{P}.
\label{eq:FLCVAR_condition}
\end{equation}
\end{corollary}

Therefore, sufficient conditions for a given partition $\mathcal{P}$ to satisfy this condition are, for example, when either $d^s=d^{s'}$ or $\beta^s=\beta^{s'}$ for all pair of scenarios $s,s'\in P$. So, for each $P\in\mathcal{P}$ we refine the partition by solving the subproblem~\eqref{eq:FLCVAR_dualsubproblem} for each $s\in P$ and group them into subsets where $\beta^s$ have the same value.

\section{Computational Experiments}\label{sec:computational}

In this section we study  computational performance of the proposed Benders adaptive-cuts method, for which we consider two strategies:
\begin{itemize}
    \item \textsc{Adaptive}: the Benders adaptive-cuts method described in Algorithm \ref{alg:gen_benders}, and 
    \item \textsc{Adaptive-Single}:  the Benders adaptive-cuts method described in Section \ref{sec:adaptive-single} in which we adapt Algorithm~\ref{alg:gen_benders} by replacing the optimality adaptive-cuts (Step \ref{alg:seq1_optcut}) by
\[ \sum_{P\in\mathcal{P}^{(t)}} p^P (h^P-T^P x)^\intercal \hat{\lambda}^P \leq \underbrace{\sum_{s\in \mathcal{S}} p^s\theta^s}_\Theta\] 
and by including this cut only when all aggregated subproblems (for all $P \in \mathcal{P}$) are feasible. Additionally,  we include the standard single cut $\sum_{s\in\mathcal{S}} p^s (h^s-T^s x)^\intercal \hat{\lambda}^s \leq \Theta$ into the master problem each time that we refine the partition $\mathcal{P}^{(t)}$ (Step \ref{alg:disggregation_step}). 
    \end{itemize}
Our computational study is conducted on three stochastic network flow problems presented in Section~\ref{sec:BendersStochNetworks}. We compare the Benders adaptive-cuts method against:
\begin{itemize}
    \item \textsc{Single}: the single-cuts Benders implementation described in  Section~\ref{sec:bendersClassic}, 
    \item \textsc{Multi}: the multi-cuts Benders implementation described in  Section~\ref{sec:bendersClassic},  
    \item  \textsc{DE}: the deterministic equivalent problem reformulation given by \eqref{eq:DE}, solved directly using an off-the-shelf solver, and \item \textsc{GAPM}: the GAPM proposed by~\citet{ramirez2021generalized} where each iteration is solved using an off-the-shelf solver of the compact problem reformulation given by \eqref{eq:DE_P}. 
\end{itemize} 
 
\paragraph{Algorithmic Considerations:} None of the aforementioned algorithms is enhanced using valid inequalities, warm starts, stabilization or additional performance improvement strategies for decomposition algorithms. 
%\EM{Several improvements can be applied to keep a partition of small size. For example, using dual stabilization or a similar technique to obtain similar duals in the case of degenerated problems, or considering a small tolerance between duals to group them, or re-constructing the partition based on the last first-stage solutions obtained by the algorithm. However, we studied the straightforward version of the algorithm.}
Besides, in all study cases the set of first-stage variables $\mathcal{X}$ is a polytope. That way, the results are useful to identify the advantages and disadvantages of the three \textsc{Adaptive} strategies, and measure the pure impact of the new methodology. 
Furthermore, the refinement of a partition $\mathcal{P}$ (Line~\ref{alg:disggregation_step} from Algorithm~\ref{alg:gen_benders}) is adapted for each problem using ideas discussed in Section~\ref{sec:BendersStochNetworks}.

Finally, the maximum running time allowed for each of the above algorithms is set to 86,400 seconds (24 hours).
All the methods were implemented using the Python programming language and using Gurobi v9.0.2 as off-the-shelf optimization solver with its default parameters. The codes and instances are available at \url{https://github.com/borelian/AdaptiveBenders}. All runs use four threads and 32 GB of RAM.  The computers employed for experiments use CentOS Linux v7.6.1810 on x86\_64 architecture, with four eight-core Intel Xeon E5-2670 processors and 128 GB of RAM.

\subsection{Results for the CPP}
\paragraph{Dataset.}
Our benchmark set for the CPP consists of the \emph{electricity planning} instances available at \url{https://people.orie.cornell.edu/huseyin/research/sp_datasets/sp_datasets.html}. They have $|V_L|=20$ source nodes and $|V_R|=50$ demand nodes. Furthermore, there are 10 resource constraints that limit the first-stage decisions.
%; in particular, they represent the set $\mathcal{X}$ of problem \eqref{eq:main}. 
The stochastic demand is a discrete probability distribution that is pairwise independent for each demand node, with a total of more than $10^{42}$ possible scenarios. For this experiment, we were able to test instances with the following numbers of sampled scenarios: 1,000, 10,000, 100,000, and 1,000,000.

% \begin{figure}[htbp]
% \centering
% \includegraphics[width=0.6\textwidth]{CP_GapTime.pdf}
% \caption{Gap over time for different numbers of scenarios} \label{fig:gaptime_scenarios}
% \end{figure}
\begin{figure}
\includegraphics[width=0.8\textwidth]{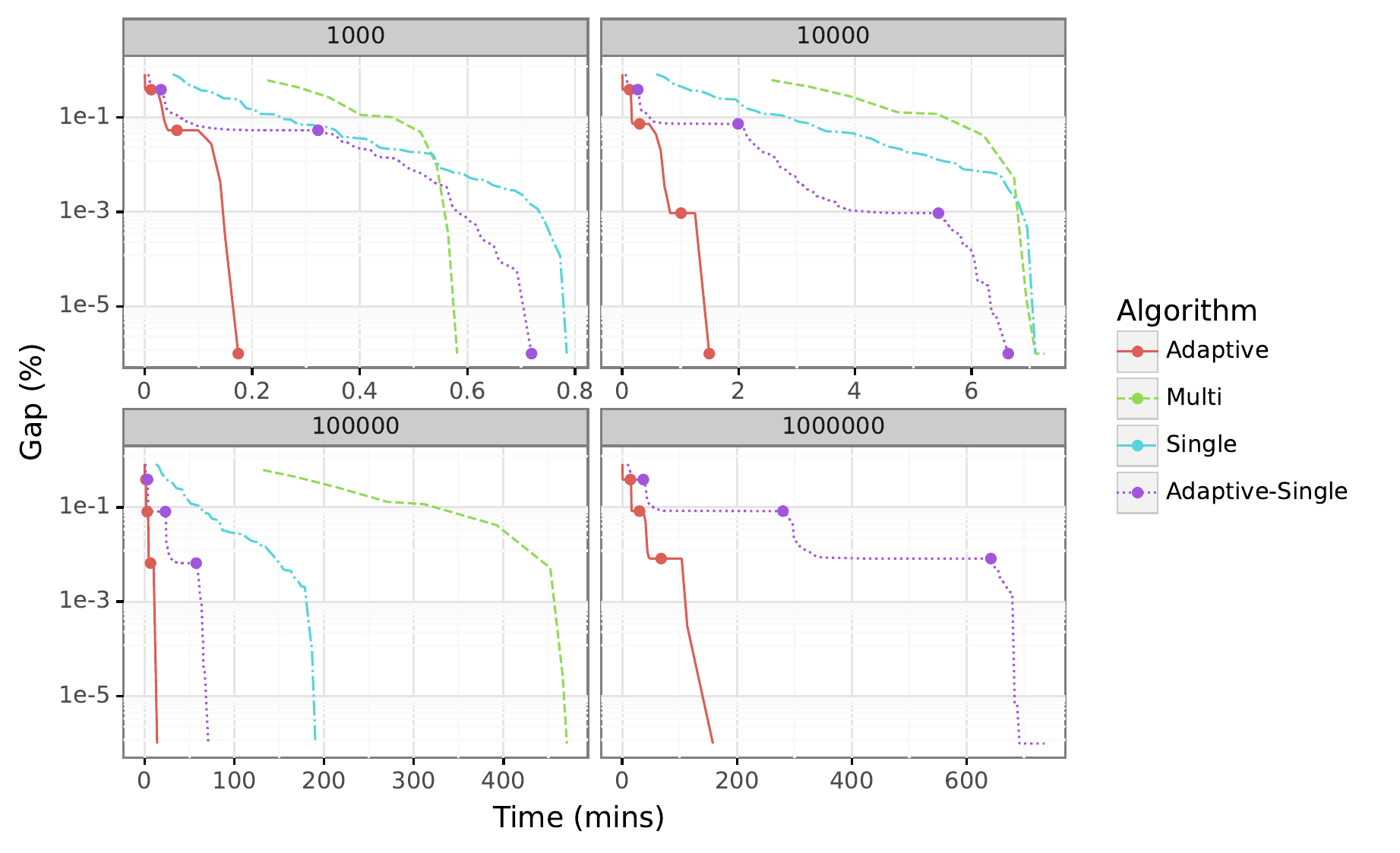}
\caption{Gap Over Time for Different Numbers of Scenarios}\label{fig:gaptime_scenarios}
\end{figure}

\paragraph{Results.}
To analyze the obtained results, we plot the percentage gap over time of the global lower bound $z_L$ obtained through the execution of our algorithm, i.e., $gap=(OPT - z_L)/|OPT|$, where
%define the \emph{gap} of an incumbent solution as the relative difference between its objective value and 
$OPT$ refers to the optimal objective value for each instance. 
Figure \ref{fig:gaptime_scenarios} shows the gap (in log-scale) over time for the different implementations (in different line-styles and colors). In the case of the \textsc{Adaptive} and \textsc{Adaptive-Single}, we also plot a dot when a refinement of the partition is done.
This figure shows that \textsc{Adaptive} not only outperforms the other methods, but also obtains better initial solutions in the early iterations. This is particularly evident for the largest instances, i.e., those with 100,000 and 1,000,000 samples, where the optimal solution is obtained much faster by \textsc{Adaptive}. %which is also the only approach able to obtain the optimal solution for the largest case, in no more than 40,000 seconds. 
Comparing the standard Benders' strategies, \textsc{Multi} is faster than \textsc{Single} when $\mathcal{S}=$1,000, but this performance is reverted when the number of scenarios increases (which is due to the master problem being overloaded with a large number of cuts separated by \textsc{Multi}). Nevertheless, neither \textsc{Multi} nor \textsc{Single} is able to deal with 1,000,000 scenarios. Not even 
%, these strategies cannot solve 
the initial iteration %(the initial master problem LP) 
can be solved by these two methods, due to the large time required to solve the 1,000,000 subproblems (recall that these two methods require  all $|\mathcal{S}|$ subproblems to be solved in order to derive Benders cuts). On the contrary, \textsc{Adaptive} and \textsc{Adaptive-Single} are able to find the optimal solution because they require to solve only $|\mathcal{P}|$ subproblems in each separation phase. Given that \textsc{Adaptive} can be interpreted as a multi-cuts implementation with respect to the given partition $\mathcal{P}$, it is not surprising that its performance is superior to that of \textsc{Adaptive-Single}. Indeed, this can be explained by two factors: (a) in each iteration stronger lower bounds are obtained due to the generation of multiple cuts, and (b) the size of $\mathcal{P}$ remains small in the final iteration, so that the master problem is not overloaded with a large number of cuts (which is sometimes a disadvantage of the multi-cuts implementation, when applied to the whole set $\mathcal{S}$, see above).

%Furthermore, for $100,000$ scenarios, \textsc{Multi} is 24 times slower than \textsc{Adaptive} (and 5.7 times slower than \textsc{Single}), and for $1,000,000$ scenarios} \textsc{Multi} is only able to run 3 iterations and the last one reaches the iteration time-limit of 6 hours due to the large number of optimality cuts \textsc{Multi} added to the problem. 

% \begin{figure}[htbp]
% \centering
% \setlength{\lineskip}{\medskipamount}
% \subcaptionbox{Number of cuts added over time\label{fig:3a}}{\includegraphics[width=0.523\textwidth]{CP_CutsTime.pdf}}\hfill
% \subcaptionbox{Partition size over time - Adaptive\label{fig:3b}}{\includegraphics[width=0.477\textwidth]{CP_PercentagePartitionSize.pdf}}\hfill
% \vspace*{2mm}\caption{Number of cuts and partition size over time for different numbers of scenarios} \label{fig:CP_cuts_part}
% \end{figure}

% \begin{figure}
% \FIGURE{\subcaptionbox{Number of cuts added over time\label{fig:3a}}{\includegraphics[width=0.523\textwidth]{CPP_Cuts.pdf}}\hfill
% \subcaptionbox{Partition size over time - \textsc{Adaptive}\label{fig:3b}}{\includegraphics[width=0.477\textwidth]{CPP_Partition.pdf}}}
% {Number of Cuts and Partition Size Over Time for Different Numbers of Scenarios \label{fig:CP_cuts_part}}{}
% \end{figure}

\begin{figure}
\subfloat[Number of cuts added over time\label{fig:3a}]{\includegraphics[width=0.523\textwidth]{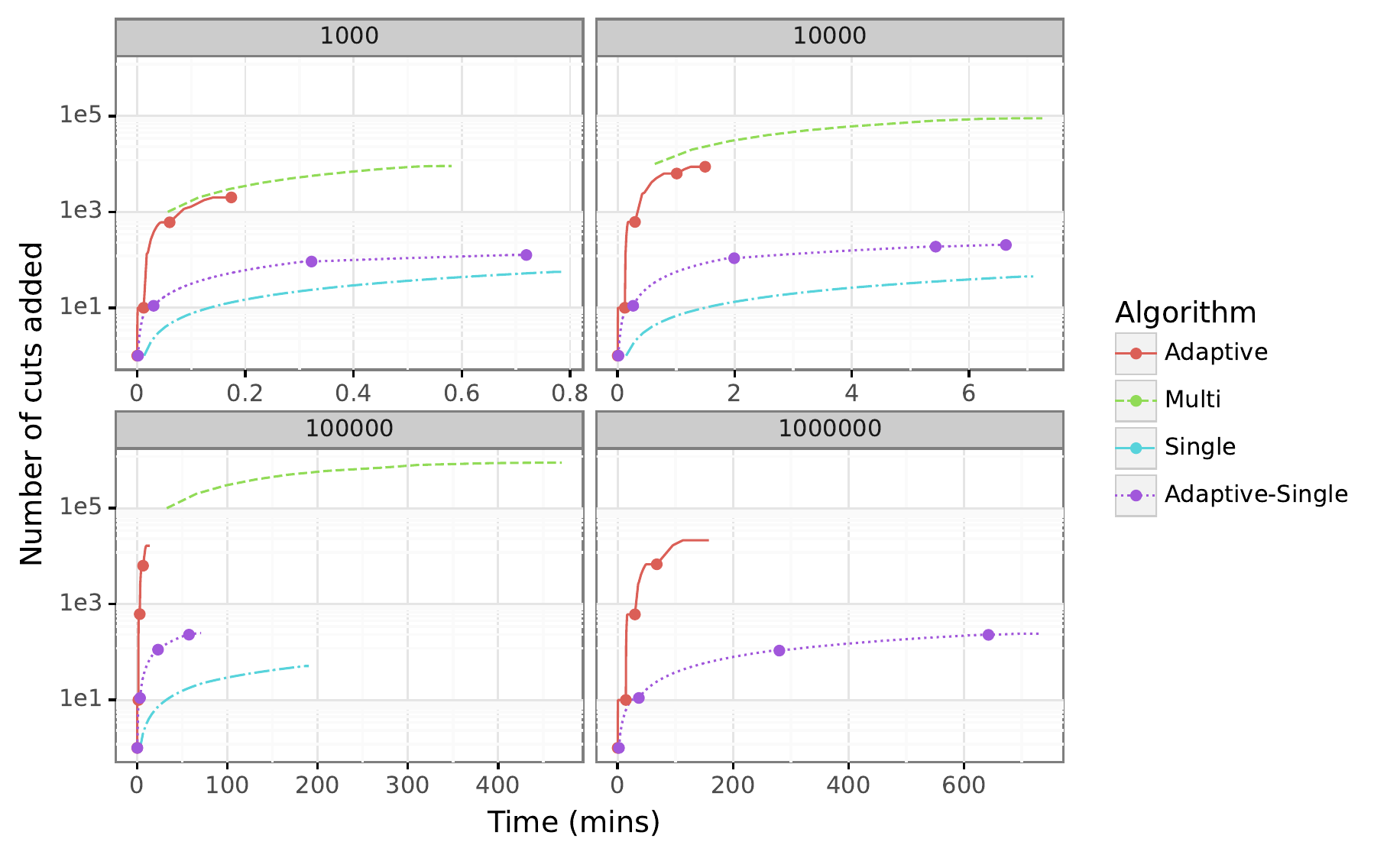}}\hfill
\subfloat[Partition size over time - \textsc{Adaptive}\label{fig:3b}]{\includegraphics[width=0.477\textwidth]{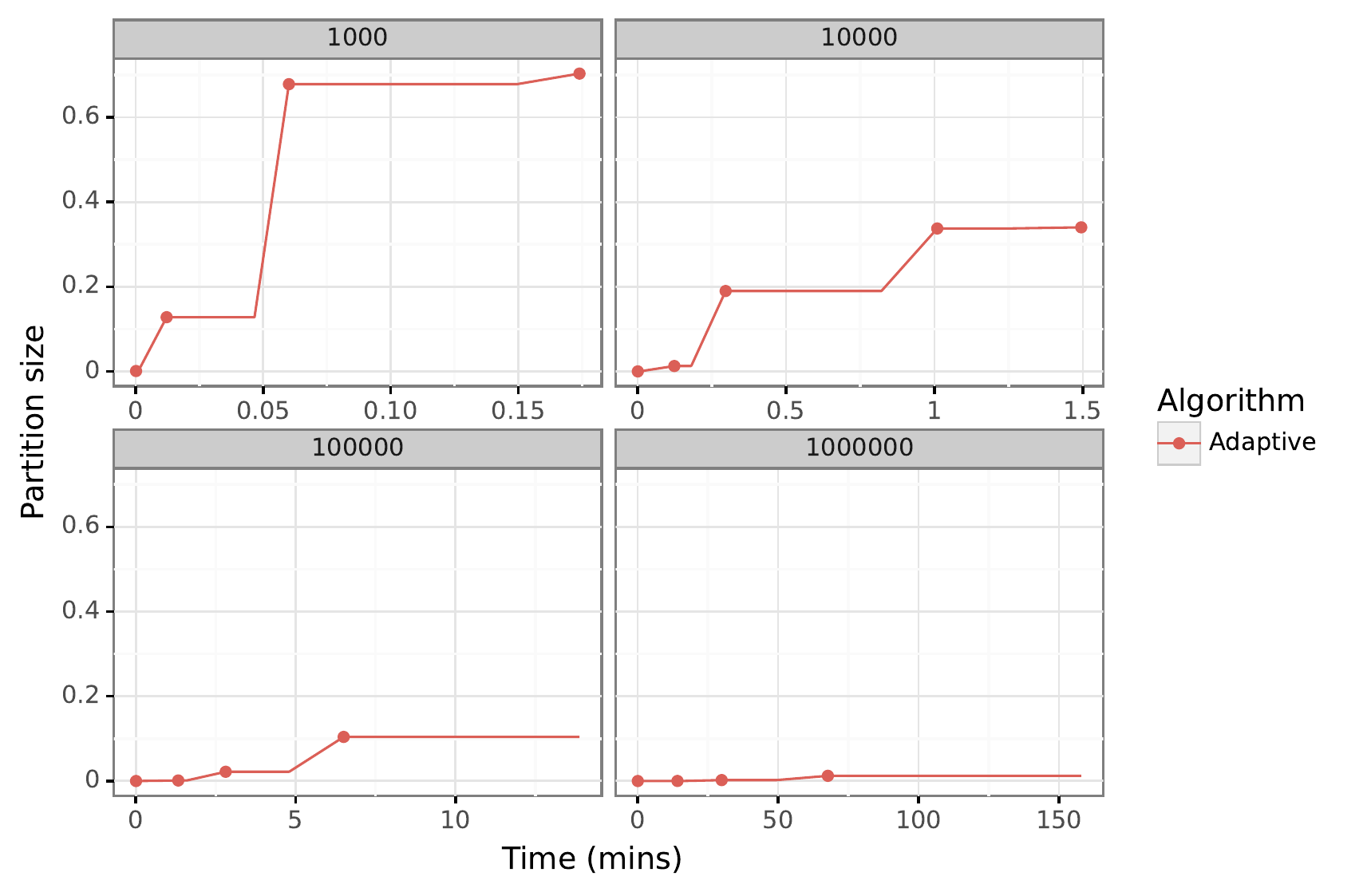}}
\caption{Number of Cuts and Partition Size Over Time for Different Numbers of Scenarios} \label{fig:CP_cuts_part}
\end{figure}

As a confirmation of these statements, we have also summarized the outcomes concerning the number of added cuts and the growth of the partition size in Figure \ref{fig:CP_cuts_part}. In Figure \ref{fig:3a}, we present the number of inserted optimality cuts over time, in log-scale. The superior performance of both \textsc{Adaptive} and \textsc{Adaptive-Single}  on very large instances is  tied  to three factors: (1) their ability to generate violated optimality cuts faster than other methods (due to the aggregation of scenarios, which allows  smaller number of subproblems to be solved, particularly in the initial iterations), (2) the fewer number of cuts required to reach the optimal solution, and (3) less subproblems to be solved even in the last iterations, where usually $|\mathcal{P}|$  remains much smaller than $|\mathcal{S}|$. By contrast, \textsc{Multi} requires more cuts, particularly for the largest instances, and \textsc{Single} adds fewer cuts but has a much slower convergence rate, requiring a longer time than \textsc{Adaptive} to reach optimality (but shorter than \textsc{Multi} for the largest instances). 
%This is due to the aggregation of scenarios of \textsc{Adaptive}, which allow  smaller problems to be solved, particularly in the initial iterations, but also results in a smaller problem in the last iterations. 
In Figure \ref{fig:3b},  the $y$-axis shows, as a percentage, the relative size of the partition ($|\mathcal{P}|/|\mathcal{S}|$) over time. This plot shows that with 1,000 scenarios, the final partition has a size of $70\%$ of $|\mathcal{S}|$, while for the remaining instances the relative size of the final  partition  keeps decreasing. For the largest instances, the final partition contains $\approx$ 12,000 scenarios, which is nearly 1.2\% of the total number of scenarios.  Hence, these instances demonstrate the 
great ability of \textsc{Adaptive} to keep small partitions over time, being able to find optimal solutions for instances  with a huge number of scenarios that remain out of reach for standard implementations of Benders method. 
%Then, the powerful results solving largest instances obey the requirement of being able to 

\paragraph{Comparison with other methods.}
Finally, we compare the performance of Benders adaptive-cuts with that of the deterministic equivalent (\textsc{DE}) and the standard Generalized Adaptive-Partition Method (\textsc{GAPM}). 

\begin{table}
	\caption{Solving Time Required by the Different Methods for the CPP}\label{tab:comparisonGAPM2}%
	\begin{small}
%{      \begin{adjustbox}{width=\textwidth,center}              \begin{tabular}{cr|rrrrrr|rrrrr|rrrrr}
      \begin{tabular}{r|rrrrrr}
            \hline
   $|\mathcal{S}|$ & \textsc{Adaptive} & \textsc{Adaptive-Single} & \textsc{Multi} & \textsc{Single} & \textsc{GAPM} & \textsc{DE}  \\ \hline
   10 & 0.29 & 0.54 & 0.30 & 1.04 & 0.57 & 0.26 \\
   100 & 2.42 & 7.00 & 3.22 & 4.75 & 6.73 & 2.41\\
   1,000 &   10.4 &    43.1  &    34.9 &    47.1 &  20.1 &   30.0 \\
  10,000 &   89.7 &   397.9  &   435.2 &   426.0 & 172.6 &  233.7 \\
 100,000 &  832.5 &  4253.9  & 28238.7 & 11416.4 & 685.5 & 2180.1 \\
1,000,000 & 9475.4 & 44126.6  &   --   & --    & 4710.3 & --\\\hline
%        \end{tabular}\end{adjustbox}}{}
\end{tabular} \end{small}
     %     \vspace{ - 05 mm}
\end{table}

Table~\ref{tab:comparisonGAPM2} presents the solving times (in seconds) required by each algorithm. 
%We notice that \textsc{Adaptive} outperforms not only the other Benders strategies but also the deterministic equivalent formulation. All algorithms perform  similarly for smaller instances. 
All algorithms perform  similarly for smaller instances. We notice that \textsc{Adaptive} outperforms the other Benders strategies for all instances. In addition, \textsc{Adaptive} has similar solving times with the deterministic equivalent formulation when the number of scenarios is smaller than 100. 
However, for larger instances, solving the DE becomes difficult, so that the largest instance cannot be solved within 24h. We also found that for the CPP the GAPM performs very well on large instances, because GAPM exploits the ability to keep small partitions over time, being faster than \textsc{Adaptive} on the largest ones. 

To show the relevance of considering a large number of scenarios for this particular problem, Table~\ref{tab:solutions} presents the optimal solution found for each problem instance obtained with a different number of scenarios. We also show the obtained objective values, and a 95\% confidence interval of the true objective value for the found solutions. It can be seen that considering less than 10,000 scenarios leads to a different optimal solution in each case, which in fact are suboptimal compared to the solution found with more than 10,000 scenarios. Also, notice that the reported objective values of the problem differ considerably from the true objective value of the solution, and that at least 10,000 scenarios are needed in order to have a reported objective value inside the 95\% confidence interval.

\begin{table}
\caption{Solutions Found for the CPP \label{tab:solutions}}%
%{      \begin{adjustbox}{width=\textwidth,center}              \begin{tabular}{cr|rrrrrr|rrrrr|rrrrr}
      \begin{tabular}{r|c|c|c}
            \hline
 \multirow{2}{*}{$|\mathcal{S}|$}   & Optimal solution & Reported  & True objective  \\
   & [$x_1,x_6,x_8,x_{10},x_{11},x_{13},x_{19}$] & objective & (95\% confidence) \\ \hline
     10 & [4,7,5,5,6,10,2] & -1527.96 & -1396.72 $\pm$ 0.80 \\
    100 & [3,6,4,4,5,11,3] & -1446.82 & -1434.80 $\pm$ 0.70 \\
   1,000 & [4,6,4,4,5,11,3] & -1447.35 & -1438.23 $\pm$ 0.72 \\
  10,000 & [4,6,4,3,5,11,3] & -1438.45 & -1438.53 $\pm$ 0.71 \\
 100,000 & [4,6,4,3,5,11,3] & -1438.12 & -1438.53 $\pm$ 0.71 \\
1,000,000 & [4,6,4,3,5,11,3] & -1438.72 & -1438.53 $\pm$ 0.71 
        \end{tabular}
	%     \vspace{ - 05 mm}
\end{table}

\subsection{Results for the SMCF}
\paragraph{Dataset.}
We tested the algorithm on the linear version of {Canad instances \textbf{R}} from~\citet{crainic2001bundle} available at  \url{http://groups.di.unipi.it/optimize/Data/MMCF.html#Canad}, but with random demands generated as in~\cite{rahmaniani2018accelerating}. The instances originally proposed by~\cite{rahmaniani2018accelerating} contain up to 100 scenarios. Since our goal was to test the performance of the new method on much larger sets of scenarios, we also generated additional instances, following the same instance generation procedure. 
We %apply the algorithm over the same graph configurations, namely, 
focus on instances with $|V|=10$ and $|E|=60$, and with the number of commodities $|\mathcal{K}| \in \{ 10, 25, 50\}$, referred to as R04 (small), R05 (medium) and R06 (large), respectively.
%,  all  having but a different number of commodities to be routed, namely,  respectively. 
Concerning the fixed costs and capacities, we focus on 5 different configurations from this dataset, denoted as $l1,l3,l5,l7,l9$, wherein the the values of $f$ and $u$ are uniform over all arcs and are given as $(f_{ij},u_{ij}) \in \{ (1,1), (10,1),(5,2), (1,8), (10,8)\}$, respectively. %As a general rule, larger values suggest tighter settings for the problem. 
%and arc capacities are $(1,1,2,8,8)$.}
Starting from a %deterministic demand for each commodity, for
 deterministic instance from this dataset, we generated 30 stochastic instances with random demands: five different levels of linear correlation ($\{0\%,20\%,40\%,60\%,80\%\}$) between the commodities are considered, and for each of these levels, seven instances are generated with  $|\mathcal{S}| \in \{ $16; 100; 1,000; 5,000; 10,000; 20,000; 50,000$\}$. 
Overall, we obtained a dataset with a total of 525 instances to be solved.

\paragraph{Results.} %Most of the results in this section are based on plots that illustrate many details of the performance and evolution of the algorithms over time. %We then need to get rid of one or more dimensions of the instances solved when analyzing particular aspects of the results.
Following the results obtained for the CPP, and after some preliminary experiments, we decided to exclude 
%On this initial discussion of the results, we omit the 
\textsc{Adaptive-Single} from the computational comparison on the SMCF. 
Our preliminary experiments have shown that the performance of \textsc{Adaptive-Single} is inferior to alternative implementations of Benders method studied in this paper.
%, and is not being able to solve even the smallest instances from this benchmark set. }

\begin{figure}
\includegraphics[width=0.9\textwidth]{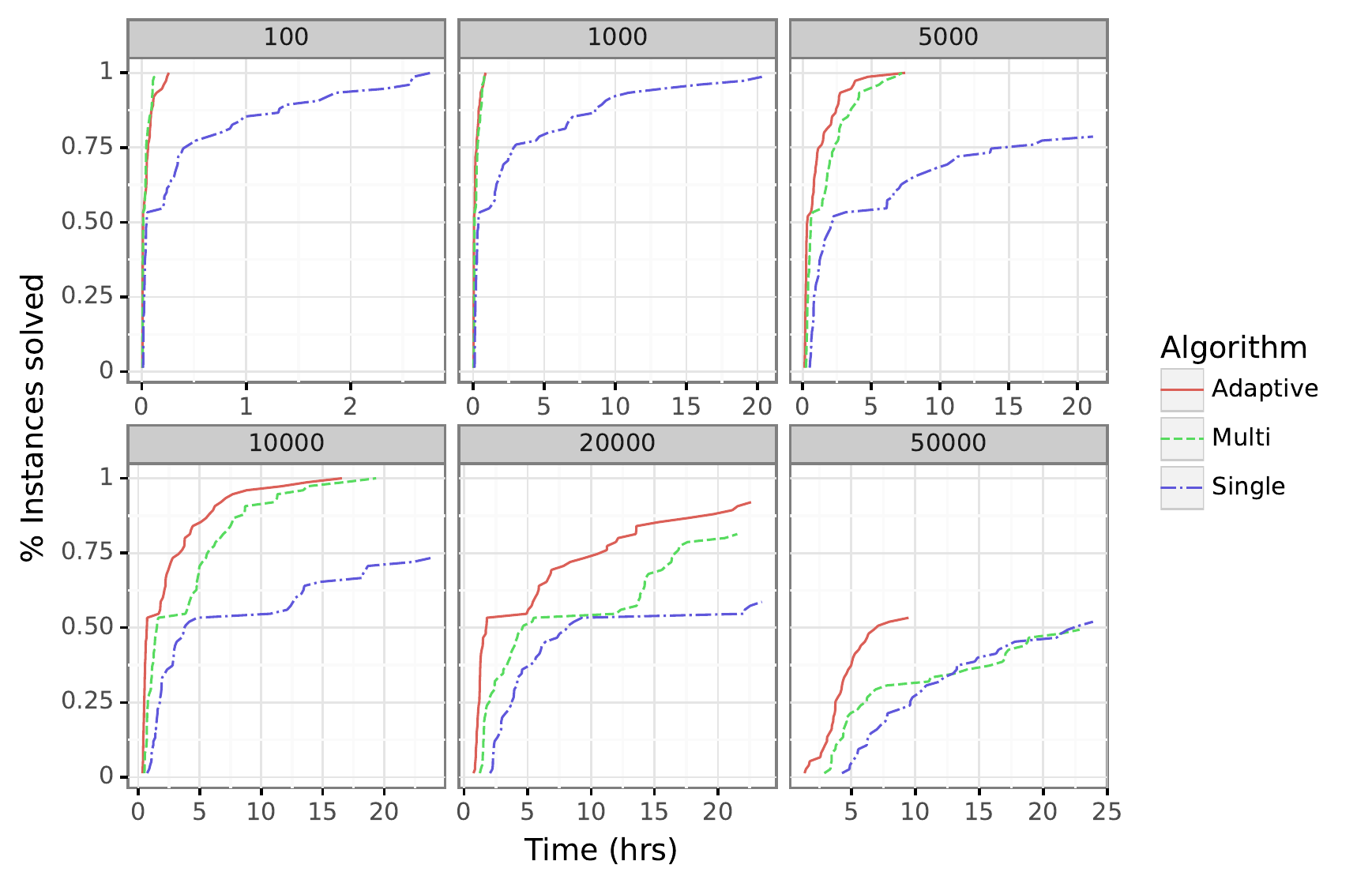}
\caption{Cumulative Performance Charts for Different Number of Scenarios}\label{fig:pprofiles}
\end{figure}

Due to the large number of instances, we first present the cumulative performance charts of solving all these instances with the different Benders strategies aggregated by the  number of scenarios. In Figure~\ref{fig:pprofiles},  each line shows the percentage of the 75 instances (per group, for a given size of $|\mathcal{S}|$, for more than 16 scenarios) solved to optimality over time for different number of scenarios. Each line corresponds to a different  Benders method, namely  \textsc{Single} in dot-dashed blue, \textsc{Multi} in dashed green and \textsc{Adaptive} in continuous red. It can be seen  that \textsc{Single} is consistently outperformed by the other methods (independently on the number of scenarios), and it is able to solve all instances only for 1,000 scenarios or less. Also, \textsc{Adaptive} clearly outperforms \textsc{Multi} when the number of scenarios is larger than 1,000, and the computational advantage of \textsc{Adaptive} is more pronounced with the increasing  number of scenarios. For 20,000 and 50,000 scenarios, \textsc{Adaptive} also solves more instances then the other two Bender's methods.

\begin{table}
	\caption{Average Solving Times, Iterations and Number of Cuts  of Each Method for Different Instances and Numbers of Scenarios}\label{tab:SFCMFP_cuts_it}%
{      \begin{adjustbox}{width=\textwidth,center}              \begin{tabular}{cr|rrrrrr|rrrrrr|rrrrrr}
            \hline
            &\multicolumn{1}{c}{}& \multicolumn{6}{c}{\textsc{Adaptive}} & \multicolumn{6}{c}{\textsc{Multi}} & \multicolumn{6}{c}{\textsc{Single}} \\
            \hline
            Inst & $|\mathcal{S}|$ & \#Ref & \#In & Time & Iter & FC & OC & \#In &  \multicolumn{2}
            {c}{Time} & Iter & FC & OC &  \#In & \multicolumn{2}
            {c}{Time}   & Iter & FC & OC \\\hline
            \multirow{6}{*}{R04} 
& 16 & 2.6 & 25 & 17.1 & 743.6 & 695 & 617 & 25 & 12.8 & (0.74) & 83.9 & 710 & 484 & 25 & 35.5 & (2.00) & 462.7 & 933 & 257\\
 & 100 & 3.2 & 25 & 28.4 & 715.3 & 872 & 1648 & 25 & 33.2 & (1.15) & 50.0 & 2283 & 1712 & 25 & 114.4 & (3.85) & 428.8 & 2608 & 235\\
 & 1000 & 3.3 & 25 & 153.3 & 752.2 & 2245 & 9997 & 25 & 264.7 & (1.70) & 38.8 & 16689 & 14410 & 25 & 911.1 & (5.66) & 385.8 & 16752 & 225\\
 & 5000 & 3.4 & 25 & 748.6 & 771.4 & 6384 & 41529 & 25 & 1540.9 & (2.00) & 35.6 & 74709 & 68595 & 25 & 4959.8 & (6.20) & 395.6 & 75877 & 222\\
 & 10000 & 3.5 & 25 & 1797.2 & 737.3 & 10309 & 79371 & 25 & 3849.2 & (2.04) & 34.2 & 140723 & 136964 & 25 & 8955.3 & (4.62) & 354.2 & 138664 & 205\\
 & 20000 & 3.6 & 25 & 4538.1 & 763.2 & 17703 & 149836 & 25 & 11490.8 & (2.39) & 35.2 & 295254 & 265130 & 25 & 19041.4 & (3.98) & 353.4 & 290202 & 206\\
 & 50000 & 3.8 & 25 & 18593.0 & 755.6 & 35164 & 333360 & 22 & 45655 & (2.21) & 30.1 & 634899 & 611159 & 24 & 49600.9 & (2.50) & 335.4 & 652731 & 203\\\hline
            \multirow{6}{*}{R05} 
& 16 & 2.3 & 25 & 66.6 & 1850.8 & 1608 & 1067 & 25 &37.1 & (0.59) & 145.7 & 1172 & 800 & 25 & 201.2 & (2.59) & 1318.3 & 2188 & 489\\
& 100 & 3.0 & 25 & 102.4 & 1722.9 & 1754 & 2492 & 25 & 89.5 & (0.95) & 85.8 & 3292 & 2534 & 25 & 881.2 & (5.01) & 1476.4 & 5079 & 462\\ 
 & 1000 & 3.0 & 25 & 407.7 & 1636.2 & 3113 & 13017 & 25 & 623.0 & (1.43) & 54.9 & 22406 & 19391 & 25 & 6212.7 & (7.85) & 1168.3 & 25441 & 392\\
 & 5000 & 3.0 & 25 & 2319.0 & 1680.7 & 7859 & 57436 & 25 & 4374.9 & (1.63) & 54.0 & 102603 & 94304 & 24 & 25334.5 & (6.44) & 908.4 & 102228 & 318\\
 & 10000 & 3.1 & 25 & 5583.0 & 1531.2 & 11908 & 106931 & 25 & 11375.8 & (1.79) & 50.8 & 196130 & 188977 & 20 & 33894.0 & (4.50) & 563.4 & 192746 & 252\\
 & 20000 & 3.1 & 25 & 16563.1 & 1537.8 & 20488 & 208087 & 25 & 34261.1 & (1.80) & 49.4 & 364854 & 375155 & 11 & 17073.4 & (2.52) & 126.8 & 256809 & 65\\
 & 50000 & 2.3 & 10 & 12351.0 & 376.2 & 27211 & 223703 & 10 & 14711 & (1.18) & 23.7 & 538463 & 284365 & 10 & 28116.5 & (2.19) & 73.4 & 525782 & 33\\\hline
\multirow{6}{*}{R06} 
& 16 & 2.3 & 25 & 288.3 & 4327.5 & 3809 & 1752 & 25 & 124.1 & (0.48) & 234.1 & 2258 & 1261 & 25 & 1105.4 & (2.70) & 2730.4 & 4825 & 1083\\
 & 100 & 3.0 & 25 & 357.3 & 3882.0 & 3631 & 3842 & 25 & 240.9 & (0.78) & 115.2 & 4647 & 4210 & 25 & 3878.1 & (6.38) & 2441.2 & 7136 & 702\\
 & 1000 & 3.1 & 25 & 1385.4 & 3912.1 & 5344 & 22135 & 25 & 1610.1 & (1.25) & 73.2 & 29781 & 30440 & 24 & 23759.1 & (9.34) & 1613.3 & 30935 & 571\\
 & 5000 & 3.2 & 25 & 7922.9 & 3745.3 & 10027 & 95988 & 25 & 11396.0 & (1.57) & 71.1 & 127143 & 145203 & 10 & 13607.8 & (3.16) & 286.3 & 85420 & 89\\
 & 10000 & 3.3 & 25 & 19306.2 & 3956.2 & 15448 & 175679 & 25 & 27714.5 & (1.53) & 63.8 & 212647 & 279003 & 10 & 24954.2 & (2.76) & 266.1 & 151486 & 95\\
 & 20000 & 2.9 & 19 & 40475.0 & 2980.1 & 21944 & 292090 & 10 & 32340.8 & (1.54) & 41.6 & 322966 & 290521 & 8 & 36365.0 & (2.69) & 184.6 & 293469 & 69\\
 & 50000 & 1.5 & 5 & 6357.0 & 131.6 & 5566 & 37159 & 5 & 15964 & (2.52) & 23.4 & 635192 & 163153 & 5 & 22115.6 & (3.53) & 42.4 & 572869 & 14\\\hline
        \end{tabular}\end{adjustbox}}
     
%     \vspace{ - 05 mm}
\end{table}

%We also analyze the feasibility and optimality cuts added by each version of the algorithm. 
Table \ref{tab:SFCMFP_cuts_it} summarizes some important aspects, including the number of instances solved to optimality (\#In), the average solving time required (Time) in minutes, the average number of iterations (Iter), the average number of feasibility (FC) and optimality (OC) cuts added. %, and the total number of cuts added per iteration (Av.Cut/It). 
We present these values for the three considered Benders methods and for different instance sizes and numbers of scenarios, with a total of 25 instances per row.
We also report the relative slow-down of  \textsc{Multi} and \textsc{Single}, respectively, compared to \textsc{Adaptive}. We divide the time required by the former ones with the time required by  \textsc{Adaptive}  for the same instance, and report the geometric mean over 25 instances. This factor is shown between parenthesis in the column Time.
Finally, for \textsc{Adaptive}, the table also shows the number of refinements of the partitions (\#Ref). 

We start by comparing the \textsc{Single} vs the \textsc{Multi} implementation. The \textsc{Single} approach requires more iterations, with a similar number of feasibility cuts and (as expected) a considerably smaller number of optimality cuts added. However, this reduction in the number of cuts does not help to solve more problems faster. In fact, \textsc{Single} can only solve all 25  instances  per group when considering a small number of scenarios. On the contrary, \textsc{Multi} adds a large number of cuts per iteration, but solves all instances per group with up to 10,000 scenarios, and is overall faster than \textsc{Single}. 

When we analyze these values for \textsc{Adaptive}, we can see that \textsc{Multi} solves a similar number of instances, with shorter solving times for 16 and 100 scenarios, but it is between 1.2 to 2.4 times slower than \textsc{Adaptive} for larger number of scenarios. Note that \textsc{Adaptive}
requires the most iterations, however each iteration of the master is solved faster, which is explained by the number of cuts added to the master problem. In fact, \textsc{Adaptive} requires considerably fewer feasibility cuts than do the other methods  because feasibility cuts are 
%required to obtain a feasible solution 
added for the aggregated scenarios, not for each single scenario (see Section \ref{sec:implementation} for more details). 
%Thus, this method has better initial bounds. On the other hand, 
Also the number of optimality cuts added by \textsc{Adaptive} is smaller than that added for \textsc{Multi}. Overall, the average number of cuts generated per iteration of \textsc{Adaptive} is considerably smaller than the respective number of \textsc{Multi} and \textsc{Single}. Indeed, for  \textsc{Adaptive}, less than 900 OC and FC are generated per iteration, regardless of the instance size (R04, R05, R06). However, the average number of cuts per iteration of \textsc{Multi} and \textsc{Single} can be as high as $\sim$47,000 and $\sim$22,000, respectively. This trade-off between more iterations but fewer cuts per iteration explains the computational advantage of   \textsc{Adaptive}, which enables more instances to be solved within the given time limit. 
Note also that for \textsc{Adaptive} the number of performed refinements (column \#Ref) is fairly small, even for a large number of scenarios. This finding also explains the good performance of \textsc{Adaptive} as this is the only step where the algorithm is required to solve all $|\mathcal{S}|$ subproblems, in contrast to the other two methods wherein the total number of subproblems must be solved in each iteration.

\begin{figure}
% \subcaptionbox{100 scenarios\label{fig:1a}}{\includegraphics[width=0.5\textwidth]{SFCMFP_TimeGap_100.pdf}}\hfill
% \subcaptionbox{$10,000$ scenarios\label{fig:1b}}{\includegraphics[width=0.5\textwidth]{SFCMFP_TimeGap_10000.pdf}}}
%\subcaptionbox{100 scenarios\label{fig:1a}}{
\includegraphics[width=0.8\textwidth]{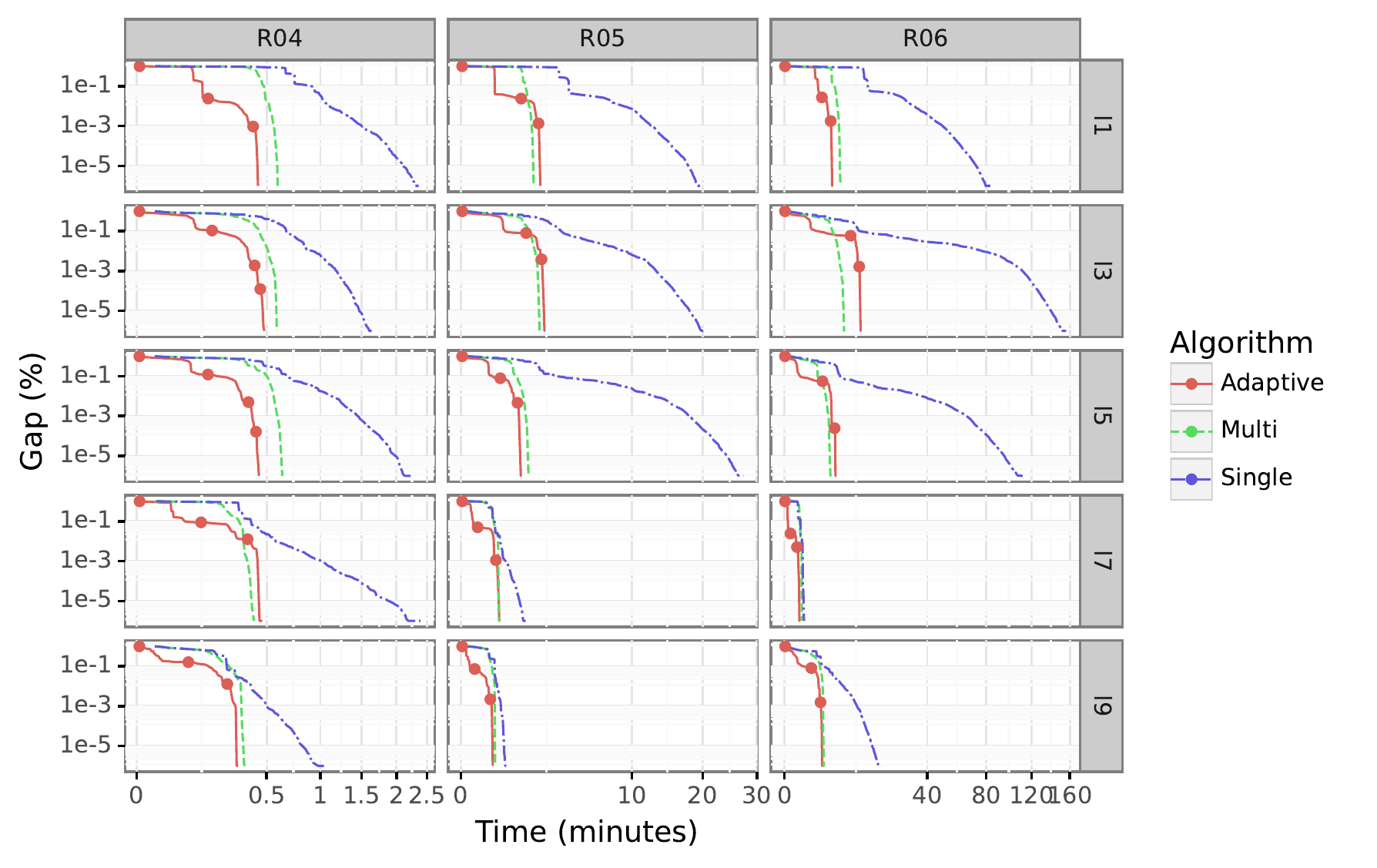}%\hfill
%\subcaptionbox{$10,000$ scenarios\label{fig:1b}}{\includegraphics[width=0.5\textwidth]{gapS10000.pdf}}}
\caption{Gap Over Time for Different Configurations (100 scenarios)}\label{fig:1_gaptime100}
\end{figure}

\begin{figure}
% \subcaptionbox{100 scenarios\label{fig:1a}}{\includegraphics[width=0.5\textwidth]{SFCMFP_TimeGap_100.pdf}}\hfill
% \subcaptionbox{$10,000$ scenarios\label{fig:1b}}{\includegraphics[width=0.5\textwidth]{SFCMFP_TimeGap_10000.pdf}}}
%\subcaptionbox{100 scenarios\label{fig:1a}}{
%\includegraphics[width=0.8\textwidth]{gapS100.pdf}}%\hfill
%\subcaptionbox{$10,000$ scenarios\label{fig:1b}}{
\includegraphics[width=0.8\textwidth]{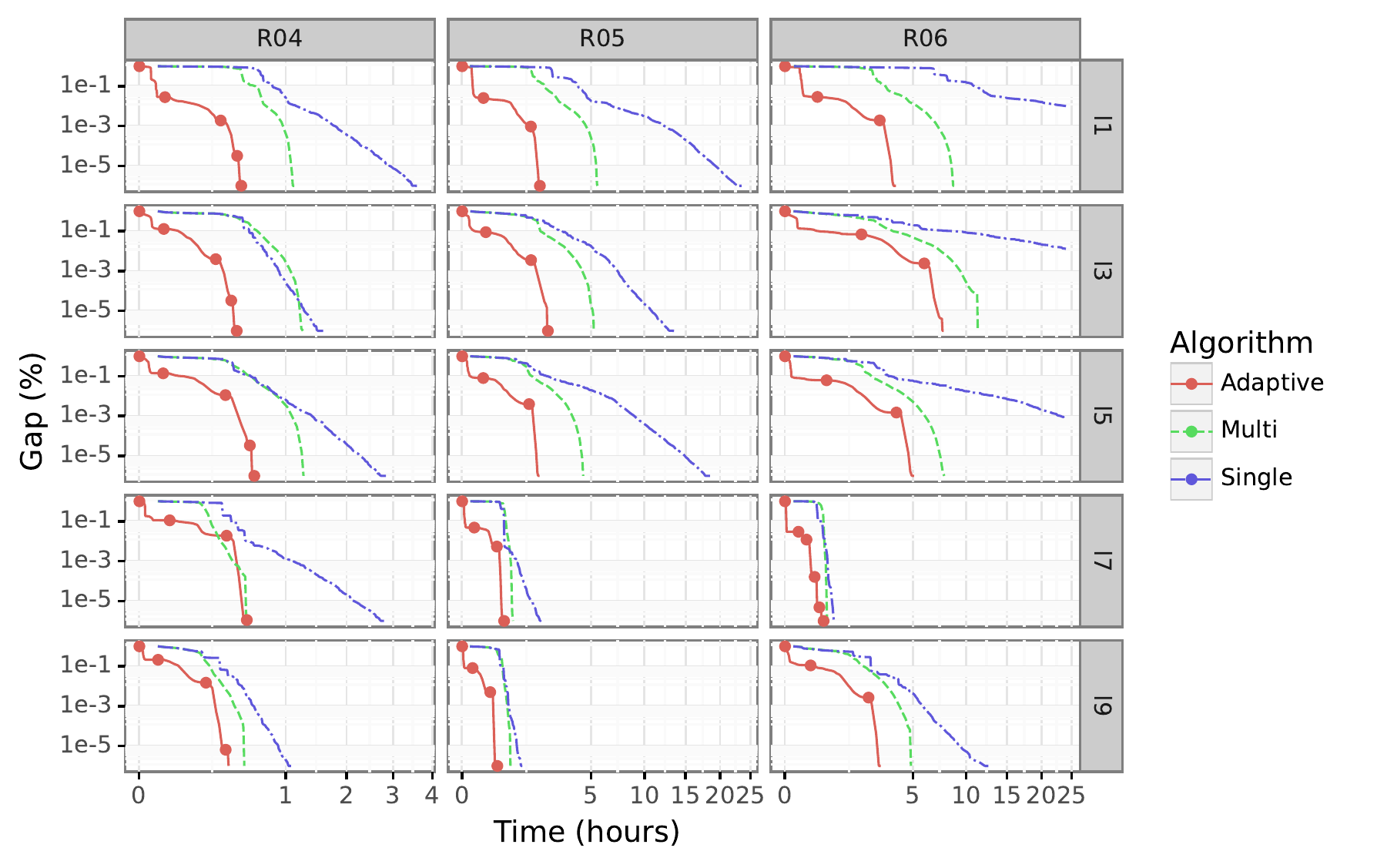}
\caption{Gap Over Time for Different Configurations (10,000 scenarios)}\label{fig:1_gaptime10000}
\end{figure}

Figures~\ref{fig:1_gaptime100} and~\ref{fig:1_gaptime10000} show the evolution of the gap over time (in log-scale) for the three methods.   For simplicity, we present only those cases with no correlation and using 100 scenarios and 10,000 scenarios (the results are  similar to the remaining configurations). %(see \url{<link_to_complete_results>} for the complete set of configurations)
For a small number of scenarios (Figure~\ref{fig:1_gaptime100}), we see that \textsc{Adaptive} and \textsc{Multi} yield similar results, being faster than \textsc{Single}, with a few cases where \textsc{Adaptive} reaches a smaller gap faster than either \textsc{Multi} or \textsc{Single}, but this is not significant given that most of the instances are solved in no more than 100 seconds. Conversely, for a large number of scenarios (Figure~\ref{fig:1_gaptime10000}), we can see how \textsc{Adaptive} has a positive impact in the early stages, obtaining a lower gap much faster than the other two Benders methods.  %In fact,  a solution close to 1\% of the optimal objective is obtained by \textsc{Adaptive} 1 or 2 orders of magnitude faster than the other methods for all cases.

\paragraph{Comparison with Other Methods}

\begin{table}
	\caption{Time Required by Adaptive-cut Benders Versus GAPM and the Deterministic Equivalent Formulation for the SMCF Problem}\label{tab:comparisonGAPM}%
%{      \begin{adjustbox}{width=\textwidth,center}              \begin{tabular}{cr|rrrrrr|rrrrr|rrrrr}
\begin{small}
      \begin{tabular}{cr|rr|rrr|rrr}
            \hline
            &\multicolumn{1}{c}{}& \multicolumn{2}{c}{\textsc{Adaptive}} & \multicolumn{3}{c}{\textsc{GAPM}} & \multicolumn{3}{c}{\textsc{DE}} \\ %& \multicolumn{3}{c}{\textsc{Adaptive-Single}}\\
            \hline
            Inst & Scen & {Av. Time} & \#In & \multicolumn{2}{c}{Av. Time} & \#In & \multicolumn{2}{c}{Av. Time}& \#In  \\% & \multicolumn{2}{c}{Av. Time} & \#In\\
            \hline
            \multirow{7}{*}{R04} 
& 16 & 17.1 & 25 & 2.9 & (0.17) & 25 & 1.39 & (0.08) & 25\\
&100   &    28.3  & 25 &    15.7 & (0.54) & 25 &     7.9 & (0.28) & 25 \\% &   146.6 & (4.94) & 25 \\
&1000  &   153.3  & 25 &   215.2 & (1.37) & 25 &   103.7 & (0.68) & 25 \\% &  1415.3 & (8.62) & 25 \\
&5000  &   748.6  & 25 &  1730.5 & (2.18) & 25 &   794.9 & (1.02) & 25 \\% &  7321.2 & (9.19) & 25\\
&10000 &  1797.2  & 25 &  5395.8 & (2.68) & 25 &  2482.6 & (1.26) & 25 \\% & 15013.3 & (7.69) & 25\\
&20000 &  4538.1  & 25 & 19979.8 & (3.70) & 25 &  9932.7 & (1.71) & 23 \\% & 31782.7 & (6.45) & 25\\
&50000 & 18593.0  & 25 & 57845.6 & (2.92) &  8 & 15562.7 & (0.71) &  1 \\% & Pend&\\
\hline
            \multirow{7}{*}{R05}
& 16 & 66.6 & 25 & 6.7 & (0.12) & 25 & 4.4 & (0.08) & 25\\
&  100 &   102.4  & 25 &    40.5 & (0.51) & 25 &    19.5 & (0.25) & 25 \\% &  1139.5 & (6.28) & 25\\
& 1000 &   407.7  & 25 &   535.7 & (1.43) & 25 &   264.8 & (0.71) & 25 \\% &  8372.9 & (11.7) & 25\\
& 5000 &  2319.0  & 25 &  4672.2 & (2.05) & 25 &  2790.0 & (1.17) & 25 \\% & 36459.4 & (10.0) & 25\\
&10000 &  5583.0  & 25 & 23609.6 & (3.53) & 24 & 10656.2 & (1.75) & 25 \\% & 27530.4 & (5.49) & 14\\
&20000 & 16563.1  & 25 & 28052.9 & (2.09) &  2 & 10138.9 & (2.41) &  3 \\% & 18726.1 & (4.04) & 10\\
&50000 & 12351.0  & 10 &      -- &  -- &   0 & -- & -- &  0 \\ % & Pend&\\
\hline
            \multirow{7}{*}{R06}
& 16 & 288.3 & 25 & 14.5 & (0.08) & 25 & 10.1 & (0.05) & 25\\
&100 &     357.3 & 25 &    90.4 & (0.38) & 25 &     42.7 & (0.18) & 25 \\% &  4561.7 & (7.00) & 25\\
&1000 &   1385.4 & 25 &  1015.8 & (0.96) & 25 &    517.9 & (0.49) & 25 \\% & 31689.3 & (12.3) & 23\\
&5000 &   7923.0 & 25 & 10942.7 & (1.68) & 25 &   5252.9 & (0.87) & 25 \\% & 23798.2 & (3.89) & 10\\
&10000 & 19306.2 & 25 & 12765.6 & (2.99) &  8 &  11793.2 & (1.52) &  8 \\% & 29368.0 & (2.97) & 8\\
&20000 & 40475.0 & 19 &  9082.5 & (3.44) &  1 &       -- &     -- &  0 \\% & Pend&\\
&50000 &  6357.0 &  5 &      -- &     -- &  0 &       -- &     -- &  0 \\% & Pend&\\
\hline
%        \end{tabular}\end{adjustbox}}{}
\end{tabular} \end{small}
     %     \vspace{ - 05 mm}
\end{table}

In Table~\ref{tab:comparisonGAPM} we compare the performance of \textsc{Adaptive}, \textsc{DE} and \textsc{GAPM}.
%present the results for the stochastic multicommodity flow problem. 
The table shows the number of instances solved to optimality for each configuration within a time-limit of 24 hours (\#In), and the average time required to solve these instances (Av. Time) in seconds.  We can see that \textsc{Adaptive} clearly outperforms the other two methods when the number of scenarios is large. In fact, while the other methods require a shorter time for the configurations with 16 and 100 scenarios, \textsc{Adaptive} is between two to four times faster when solving the problems with more than 1,000 scenarios. Moreover, \textsc{Adaptive} can solve all instances from R04 and R05, and almost all instances of the R06 group with up to 20,000 scenarios.  On the contrary, even if \textsc{GAPM} can solve more instances than \textsc{DE}, none of them is able to solve a single instance of R05 and R06 with 50,000 scenarios. This can be explained by the size of the problems that need to be solved for the alternative methods. In fact, even though the GAPM exploits the idea of scenario  aggregation, in the final iterations the problems became too big if the final partition contains a large number of scenarios. Indeed, the final partition required by the GAPM in most of the cases is very large ($97.6\%$ in median), which prevents it from finding the optimal solution. On the contrary, \textsc{Adaptive} manages to deal with such large partitions, as it draws the advantage of decomposition, and does not solve the final partition as a compact model, but uses a dynamic cut generation instead.  This shows that the advantage of \textsc{Adaptive} does not rely on finishing with a small partition (as we saw for the CPP experiments), making this method a promising alternative for a broad class of stochastic problems.

\subsection{Results for the FL-CVaR problem}
\paragraph{Dataset.} We use the classic capacitated warehouse location instances from ORLib~\citep{ORlib} available at \url{http://people.brunel.ac.uk/~mastjjb/jeb/orlib/capinfo.html}. In particular, we tested the algorithm on the instances \emph{cap41-44}, \emph{cap61-64} and \emph{cap71-74} which consider $|I|=16$ facility locations and $|J|=50$ clients. The only differences between these instances are the capacities and the installation costs of the facilities.  For the random demands, we construct scenarios by sampling $d_j^s$ uniformly between 0 and the original demand $d_j$ of each instance. The number of sampled scenarios is $|\mathcal{S}|\in\{100; 500; 1,000; 2,500\}$, and for each case we generate ten instances with different sampled demands. We set the risk-aversion level at $\sigma=90\%$.

\paragraph{Computational implementation.} Since we are dealing with a MIP problem, to avoid the excessive inclusion of cuts during the branch-and-bound process, we only add the corresponding Benders cuts when a new feasible incumbent solution has been found by the solver, i.e., using a \emph{lazy-cut callback}. Also, we keep a global partition $\mathcal{P}$ over the whole branch-and-bound tree, which could be refined by the new incumbent solutions found on different branches. 

\paragraph{Results.} We compare the different Benders strategies (\textsc{Single}, \textsc{Multi} and \textsc{Adaptive}) with the deterministic equivalent (\textsc{DE}) formulation of the problem. As before, we exclude \textsc{Adaptive-Single} due to its poor performance. Figure~\ref{fig:FLCVAR_pp} shows the cumulative performance charts of the different strategies for the different number of scenarios. Each line shows the percentage of the 120 instances solved to optimality over time. More details are also presented in Table~\ref{tab:FLCVAR}.

\begin{figure}
\includegraphics[width=0.9\textwidth]{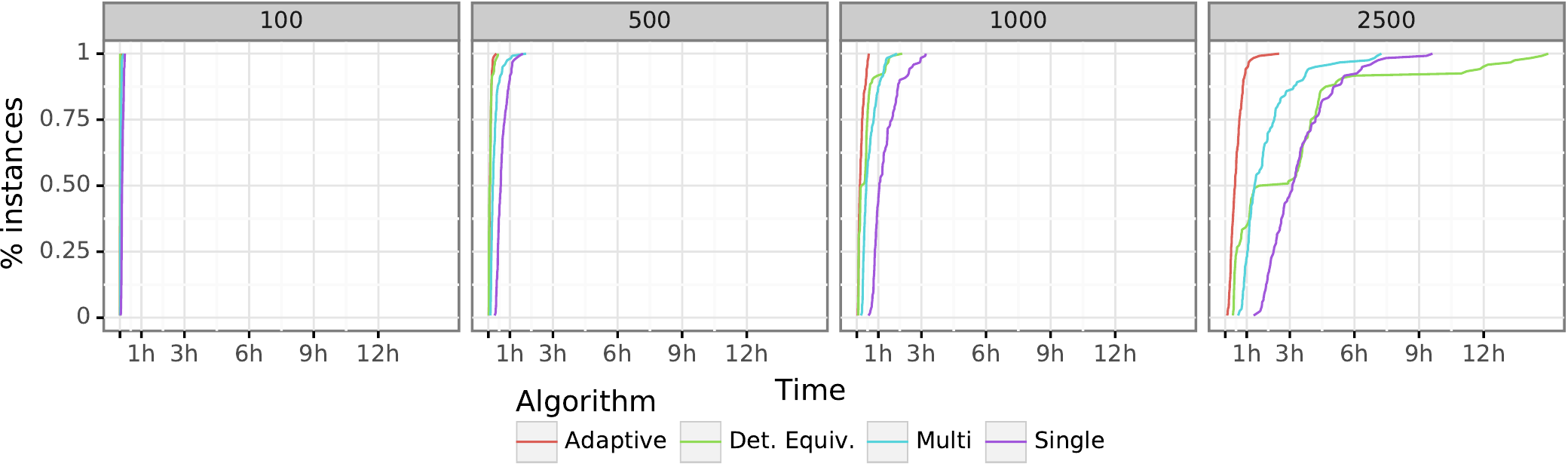}
\caption{Cumulative Performance Charts for Different Number of Scenarios for FL-CVaR problem\label{fig:FLCVAR_pp}}
\end{figure}

\begin{table}
\caption{Average Solving Times and Number of Cuts  of Each Method for Different Instances and Numbers of Scenarios for FL-CVaR problems\label{tab:FLCVAR}}%

    %  \begin{adjustbox}{width=\textwidth,center}      
{\begin{adjustbox}{width=\textwidth,center}
\begin{tabular}{lr|rrr|rlr|rlr|rl}
            \hline
            & & \multicolumn{3}{c}{\textsc{Adaptive}} & \multicolumn{3}{c}{\textsc{Multi}} & \multicolumn{3}{c}{\textsc{Single}} & \multicolumn{2}{c}{\textsc{DE}} \\
            \hline
            Inst & $|\mathcal{S}|$ & Time & \#Ref & OC &  \multicolumn{2}{c}{Time}& OC & \multicolumn{2}{c}{Time}& OC & \multicolumn{2}{c}{Time} \\\hline
            \multirow{4}{*}{cap4x} 
     & 100  &       81.0 &      8.6 &      14726 &      201.8 &  (2.73) &      24056 &      421.5 &   (5.57) &       2856 &       13.2 &    (0.20) \\
      & 500  &      361.5 &     15.1 &      53183 &     1223.9 &   (3.30) &     129223 &     2285.7 &   (6.68) &       3371 &      105.3 &   (0.33) \\
      & 1000 &      823.1 &     18.2 &     106140 &     2363.9 &  (3.04) &     229660 &     5037.7 &   (6.82) &       3643 &      311.0 &   (0.46) \\
      & 2500 &     2020.4 &     23.7 &     208466 &     7203.3 &  (3.47) &     587637 &    13013.6 &    (6.50) &       3818 &     1776.8 &   (0.94) \\ \hline
      \multirow{4}{*}{cap6x}
      & 100  &       63.9 &      9.0 &      10602 &      181.5 &  (2.77) &      23881 &      495.6 &    (8.10) &       3223 &       22.6 &   (0.38) \\
      & 500  &      412.5 &     16.0 &      49268 &     1292.4 &  (2.87) &     125685 &     2668.0 &   (7.09) &       3745 &      338.2 &   (0.88) \\
      & 1000 &      937.4 &     19.1 &      94975 &     2641.1 &  (2.79) &     235090 &     5508.7 &   (6.13) &       3785 &     1404.3 &   (1.54) \\
      & 2500 &     2449.7 &     24.4 &     190433 &     9020.4 &  (3.42) &     587653 &    13861.3 &   (5.87) &       3612 &    11667.9 &   (4.92) \\ \hline
      \multirow{4}{*}{cap7x}
      & 100  &       28.8 &      9.2 &       5345 &       90.6 &  (3.25) &      13696 &      337.9 &  (12.2) &       2417 &       27.2 &   (0.95) \\
      & 500  &      168.1 &     16.0 &      22435 &      580.6 &  (3.57) &      69124 &     1875.8 &  (11.5) &       2635 &      550.2 &   (2.74) \\
      & 1000 &      354.2 &     20.0 &      42240 &     1319.3 &  (3.88) &     138972 &     3452.2 &  (10.1) &       2627 &     2453.0 &   (5.56) \\
      & 2500 &     1098.1 &     24.6 &      95468 &     4074.4 &  (4.04) &     324071 &    10227.2 &   (9.95) &       2956 &    19975.8 &  (14.4)\\\hline
\end{tabular}\end{adjustbox}}
\end{table}

When $|\mathcal{S}|=100$, all strategies perform similarly solving all instances in less than 10 minutes. In this case, \textsc{DE} is faster than all Benders strategies. As expected, when the number of scenarios increases,  \textsc{DE} becomes slower requiring considerably more time than Benders strategies to solve all instances, in particular for $|\mathcal{S}|=2,500$.  Interestingly, this increment on the solving times depends on the instance family. While \textsc{DE} is very efficient for \emph{cap41-cap44}, even with a large number of scenarios, its performance is the worst among all methods for \emph{cap71-cap74} when $|\mathcal{S}|=2,500$.

Among the different Benders strategies, \textsc{Single} is the slowest and \textsc{Adaptive} is the fastest for all instances and all the number of scenarios studied. It can be seen that the number of optimality cuts added by \textsc{Single} is smaller, which can explain the slow convergence to the optimal solution. On the other hand, \textsc{Adaptive} requires approximately half of the number of optimality cuts required by \textsc{Multi}, so the master problems solved at each iteration are smaller, which can explain the better performance of \textsc{Adaptive}. Note that compared with the other two problem studied above, \textsc{Adaptive} requires more refinements of the partition to reach the optimal solution. This can be explained by the binary nature of the first-stage decision, because different incumbent solutions found during the branch-and-bound process of the solver refine the partition independently. Nevertheless, the size of the final partition does not grow too much. In fact, the final partition needed by \textsc{Adaptive} to prove the optimality of the solution ranges between $\approx 22\%$ of $|\mathcal{S}|$ for 100 scenarios to $\approx 13\%$ for 2,500 scenarios. This is somehow expected, because given the $\sigma=90\%$ risk-level of the CVaR objective, there are roughly speaking $\approx 10\%$ of scenarios that are the most relevant ones for finding the optimal solution.

To show the relevance of considering a large number of scenarios for this problem, we analyze the solutions found by the model. Recall that we sample ten different problems for each instance. When $|S|=100$, almost all instances obtain different optimal solutions (up to 4 different optimal solutions for \emph{cap41-cap44}). On the contrary, when $|S|=2,500$, almost all instances obtain the same optimal solution for the ten samples of the customers' demand, showing that a large number of scenarios is required to correctly model the problem. A similar behavior is observed with respect to the objective value of the optimal solution found. Figure~\ref{fig:objval_FLCVAR} shows the dispersion of the objective value among the ten sampled demands for each instance and for the different number of scenarios $|\mathcal{S}|$. We conclude that adding more scenarios is not only required to  consistently obtain the same optimal solution, but it is also needed to obtain a better estimation of the true objective value of these solutions.  

\begin{figure}
\includegraphics[width=0.9\textwidth]{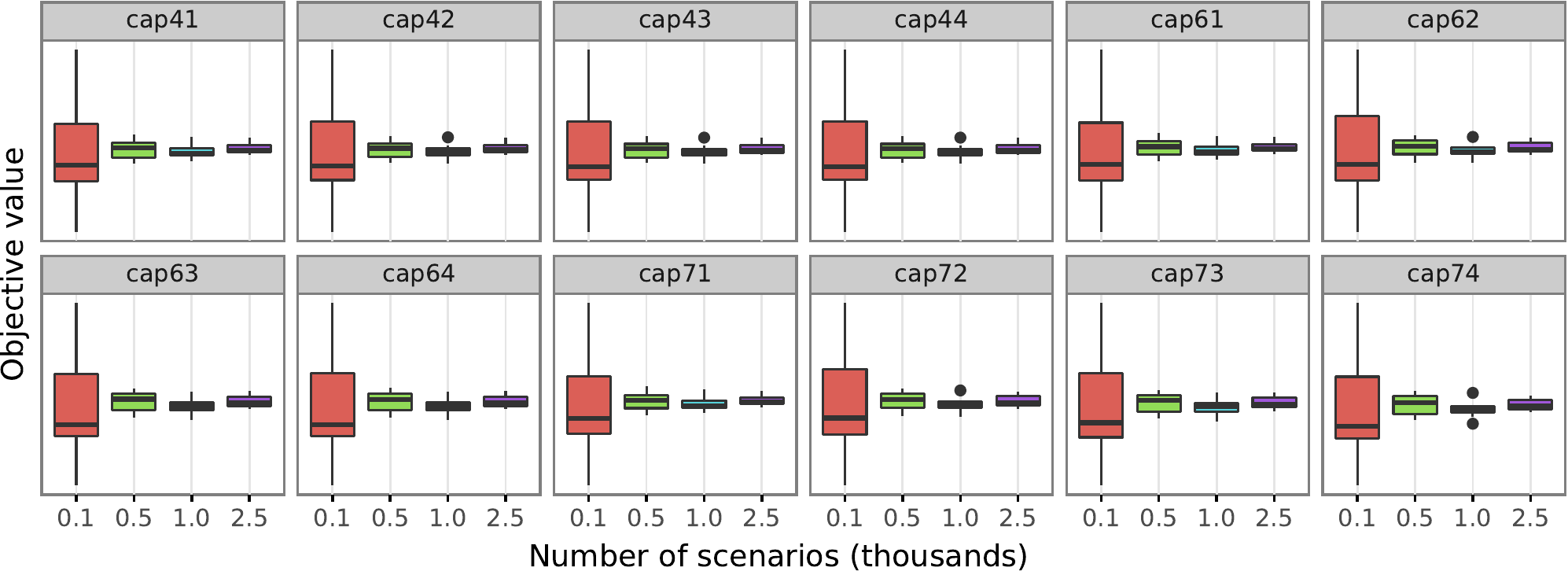}
\caption{Optimal objective value for FL-CVaR problems\label{fig:objval_FLCVAR}}
\end{figure}

\section{Concluding Remarks}\label{sec:conc}

We proposed a new  Benders adaptive-cuts method based on the GAPM for two-stage stochastic problems. 
We conducted an extensive computational analysis to highlight both the strengths and weaknesses of the proposed approach,  by focusing on stochastic network flow problems. The results show that the performance of Benders adaptive-cuts is considerably better than that of its two counterparts based on a separation of multi-cuts and  single-cuts, respectively. This superior performance is particularly pronounced in the early iterations of the cutting plane algorithms. For the stochastic multicommodity flow problem, the Benders adaptive-cuts method tends to perform similarly to Benders multi-cuts in the long term because proving the optimality frequently requires almost complete disaggregation of the set of scenarios. This does not occur for the capacity planning problem, where the  Benders adaptive-cuts method significantly outperforms the other two variants. Moreover, for very large instances of the CPP with a million of scenarios,  the final partition does not exceed 2\% of the total number of scenarios. 
For the FL-CVaR problem, Benders adaptive-cuts method can be an order of magnitude faster than the deterministic equivalent and the single-cuts approach, and it significantly reduces the number of cuts and the solving times of the multi-cuts approach as well, even for a small to moderate number of scenarios. 

Overall, the Benders adaptive-cuts method is shown to outperform its multi-cuts and single-cuts counterparts, due to to the following major factors: (1) its ability
to generate violated optimality and feasibility cuts faster than the other two methods (due to the aggregation of scenarios, particularly in the initial iterations), and (2)
the fewer number of cuts required to reach the optimal solution. The latter effect is amplified for problems whose size of the final partition remains relatively small compared to the total number of scenarios. 

We recall that in our study we did not apply any of the standard Benders decomposition algorithmic enhancements, leaving open to the reader the possibility to apply our methodology to particular problems where acceleration techniques can increase the global efficiency of Benders decomposition. 
Several improvements can be applied to keep a partition of small size. For example, using dual stabilization or a similar technique to obtain similar duals in the case of degenerated problems, or considering a small tolerance between duals to group them, or re-constructing the partition based on the last first-stage solutions obtained by the algorithm. 
When it comes to extending theory and methodology of our results, it would be interesting to study convex two-stage stochastic optimization problems (notably, with a non-linear but convex objective function in the recourse).

\bibliographystyle{plainnat}
\bibliography{biblio_benders_gapm.bib}
\end{document}